\documentclass{article}
\pdfoutput=1

\usepackage{amsmath,amssymb,graphicx,bbm}
\usepackage{amsthm,verbatim}
\usepackage{mathrsfs,mathtools}
\usepackage[footnotesize,bf]{caption}
\usepackage[left=1.25in,right=1.25in,top=1in]{geometry}
\usepackage{hyperref}

\newcommand{\edits}[1]{{#1}}
\renewcommand{\edits}[1]{#1}
\newcommand{\editsbrittle}[1]{{#1}}
\renewcommand{\editsbrittle}[1]{#1}

\usepackage{lipsum}
\usepackage{amsfonts}
\usepackage{graphicx}
\usepackage{epstopdf}
\usepackage{algorithm}
\usepackage{algpseudocode}
\ifpdf
  \DeclareGraphicsExtensions{.eps,.pdf,.png,.jpg}
\else
  \DeclareGraphicsExtensions{.eps}
\fi

\newcommand{\TheTitle}{Allocation strategies for high fidelity models in the multifidelity regime}

\title{{\TheTitle}\thanks{D.J.~Perry and R.M.~Kirby acknowledge that their part of this research was sponsored by ARL under Cooperative Agreement Number W911NF-12-2-0023. A.~Narayan is partially supported by AFOSR FA9550-15-1-0467 and DARPA EQUiPS N660011524053. R.~Whitaker acknowledges
support from DARPA TRADES HR0011-17-2-0016.
The views and conclusions contained in this document are those of the authors and should not be interpreted as representing the official policies, either expressed or implied, of ARL, DARPA or the U.S. Government. The U.S. Government is authorized to reproduce and distribute reprints for Government purposes notwithstanding any copyright notation herein. 
 }}

\author{%
Daniel J. Perry\thanks{School of Computing and SCI Institute, University of Utah, Salt Lake City, UT
    (\url{dperry@cs.utah.edu}),}
  \and
Robert M. Kirby\thanks{School of Computing and SCI Institute, University of Utah, Salt Lake City, UT
    (\url{kirby@cs.utah.edu}),
    }
	\and
Akil Narayan\thanks{Department of Mathematics and SCI Institute, University of Utah, Salt Lake City, UT
    (\url{akil@sci.utah.edu}),
		}
	\and
Ross T. Whitaker\thanks{School of Computing and SCI Institute, University of Utah, Salt Lake City, UT
    (\url{whitaker@cs.utah.edu}),
		}
}

\usepackage{amsopn}

\usepackage{amsfonts}
\usepackage{graphicx}

\usepackage{color}

\newcommand{\R}{\mathbb{R}} 
\newcommand{\N}{\mathbb{N}} 

\newcommand{\BigO}{\mathcal{O}} 

\usepackage{mathdefs}
\usepackage{mathtools}

\newcommand{\romg}{\mathrm{g}}
\newcommand{\romd}{\mathrm{d}}

\newcommand{\eedits}[1]{{#1}}

\begin{document}

\maketitle

\begin{abstract}
We propose a novel approach to allocating resources for expensive simulations of high fidelity models when used in a multifidelity framework. 
Allocation decisions that distribute computational resources across several simulation models become extremely important in situations where only a small number of expensive high fidelity simulations can be run.  
We identify this allocation decision as a problem in optimal subset selection, and subsequently regularize this problem so that solutions can be computed. Our regularized formulation yields a type of group lasso problem \eedits{that has been studied in the literature to accomplish subset selection}. Our numerical results compare performance of \eedits{algorithms that solve the group lasso problem} for algorithmic allocation against a variety of other strategies, including those based on classical linear algebraic pivoting routines and those derived from more modern machine learning-based methods. We demonstrate on well known synthetic problems and more difficult real-world simulations that \eedits{this group lasso} solution to the relaxed optimal subset selection problem performs better than the alternatives. 

\end{abstract}

%

\section{Introduction}
A persistent challenge in uncertainty quantification is estimating the effect of uncertain parameters on quantities of interest.
A common approach to understanding the effect of a parameter is to evaluate an ensemble of simulations for various parameter values.
This approach is reasonable when multiple runs of a simulation model or experiment are easily obtained.
Unfortunately, a simulation model or discretization that is more true-to-life, or has \emph{higher fidelity}, requires additional computational resources due to, for example, increased number of discrete elements or more expensive modeling of complex phenomena.
For these reasons a \emph{high-fidelity} model simulation can incur significant computation cost, and repeating such a simulation for a sufficient number of times to understand parameter effects can quickly become infeasible.

Recent work has introduced an effective solution to this problem by using multiple fidelities of simulation models where less-expensive, lower-fidelity versions of the high-fidelity model are used to learn the parametric structure of a simulation.
This structure is utilized to choose a small parameter ensemble at which high-fidelity runs are assembled, providing insight into the finer details and effects of the simulation parameters \cite{narayan2013stochastic, zhu2014computational}.
One important decision in this allocation process is which parameter values should be used in the less costly low- and medium-fidelity simulations, and which values are worth the more costly high-fidelity simulation.
This decision becomes extremely important for situations where it is physically impossible to run the high-fidelity simulation more than a small number of times, say $\mathcal{O}(10)$ times.

By assuming the simulation realizations are elements in a Hilbert space, previous work uses a Gram matrix associated to a parameter ensemble to learn the parametric structure. Standard linear algebra tools, such as the Cholesky decomposition, are utilized to rank ``important" parameter values where high-fidelity simulations are run.

The multifidelity situation is special among subset selection problems since we cannot add additional elements to a high-fidelity subset due to the significant marginal cost of additional computation.  Thus choosing the \emph{best} subset becomes quite critical in order to effectively utilize a very small number of high-fidelity simulations. While the classical linear algebraic approaches mentioned previously perform reasonably well, these approaches ignore the extensive work done in subset selection in a more general context in both the data analysis and machine learning literature \cite{cheng2005on,boutsidis2009improved,bach2013sharp,boutsidis2014near,papailiopoulos2014provable,altschuler2016greedy}.

Our main contributions in addressing this problem are as follows:
\begin{itemize}
  \item Our work includes substantial experimental comparisons with a broad range of subset selection algorithms. These include classical methods such as randomized sampling, single-layer Gaussian processes, and pivoted numerical linear algebra routines, along with more modern machine learning-based strategies such as leverage sampling and neural networks. We conclude that a particular group orthogonal matching pursuit (GOMP) algorithm yields superior results for the problems we consider.
  \item \edits{We apply existing GOMP approximation theory to a single-fidelity approximation problem, providing justification for the competitive performance of GOMP compared to other algorithms in the context of subset selection. Our analysis does not directly apply to the multifidelity method we investigate, but provides a first step toward understanding group matching pursuit approaches for this multifidelity strategy.}
  \item We empirically show that the GOMP multifidelity procedure is effective when applied to a variety of nontrivial large-scale problems in computational science, such as compressible fluid dynamics and structural topology optimization. In particular, we observe that GOMP-based methods yield superior results in almost every situation we have tried.
\end{itemize}
\edits{We note that there are many alternative approaches to tackling the multifidelity problem, and these alternatives are substantially different in scope, applicability, and goals. For example, \cite{ng_multifidelity_2012} uses additive corrections to combine high-fidelity and low-fidelity quantities of interest into predictions; \cite{le2013multi,perdikaris2015multi} uses a multilevel Gaussian process strategy to build Gaussian process predictors from models at different hierarchies; \cite{peherstorfer_optimal_2016} uses a multilevel Monte Carlo prediction strategy for expectations of scalar quantities of interest to devise a resource allocation strategy. In contrast to these approaches, our strategy uses low-fidelity parametric variation to guide allocation of high-fidelity effort, but does not require any transformation operators between levels (e.g., such as additive corrections). A direct comparison between our strategy and alternative multifidelity approaches is beyond the scope of this paper, and so we leave a such a comparison for future work.}

\section{Previous work}
We will now consider the multifidelity proxy model introduced in \cite{narayan2013stochastic} and further studied in \cite{zhu2014computational}.
The authors of \cite{narayan2013stochastic} proposed using samples from a low-fidelity model to inform the parameter selection of the much more costly high-fidelity samples and to compute the reconstruction weights of the high-fidelity proxy model based on those low-fidelity samples.
The weights are computed from the Gramian (also known as the Gram matrix or inner product matrix) of the low-fidelity samples, and so the resulting model can be considered a non-parametric model with weights derived from the low-fidelity Gramian. For simplicity we will consider only a two-level (``bifidelity") situation with only a single low- and high-fidelity model. However, this procedure can be applied to many levels \cite{narayan2013stochastic}.

\subsection{Gramian non-parametric multifidelity algorithm}
We consider the following general multifidelity setup: Let $z \in D \subset \R^q$, $q \geq 1$, be a common parametric input into two simulations models of differing fidelities. We use $u^L$ and $u^H$ to denote the outputs of the low- and high-fidelity models, respectively, with $u^L(z)$ a vector-valued output of the low-fidelity model evaluated at parameter value $z$. The models produce outputs in vector spaces $V^L$ and $V^H$, respectively. To summarize our notation:
\begin{align*}
  u^L &: D \rightarrow V^L, & u^H &: D \rightarrow V^H.
\end{align*}
The vectors $u^L$ and $u^H$ are assumed to represent any spatial/temporal effects of interest. The multifidelity regime occurs when $u^H$ is more faithful to reality than $u^L$, but requires greater computational effort to simulate. That is, 
\begin{align}\label{eq:mf-setup}
  \left\| u^H(z) - u(z) \right\| &\ll \left\| u^L(z) - u(z) \right\|, & \mathrm{Cost}\left( u^H(z) \right) \gg \mathrm{Cost}\left(u^L(z)\right),
\end{align}
where $u(z)$ represents reality\footnote{We are intentionally being vague in defining $u(z)$ and the norms in \eqref{eq:mf-setup}. Such relations can be established through formal means, such as convergence of discretizations of mathematical models, or through more informal means, such as expert belief or knowledge in superiority of the high-fidelity model.}, and $\mathrm{Cost}(\cdot)$ is the requisite computational burden of evaluating the argument. In practice the assumptions \eqref{eq:mf-setup} frequently imply that the dimensions of the vector spaces satisfy
\begin{align*}
  \dim V^L = d^L \ll d^H = \dim V^H,
\end{align*}
but it is not necessary to assume this and the multifidelity algorithm does not require this assumption. \edits{A concrete example of the above situation is the following: Consider Poisson's equation on a domain $\Omega \subset \R^2$:
  \begin{align*}
    -\triangle_x u(x,z) &= f(x, z), & x &\in \Omega,;\;\ z \in D,
  \end{align*}
  where $\triangle_x$ denotes the Laplacian with respect to the $x$ variable. One may compute appropriate solutions to this problem with finite element methods, and suppose that $u^L$ and $u^H$ are two finite element solutions computed on different meshes: $u^L$ is computed on a relatively coarse mesh with mesh size parameter \eedits{$T$}, and $u^H$ is computed on a relatively fine mesh with mesh size parameter \eedits{$t \ll T$}. Since the total number of elements on a mesh in two dimensions scales like \eedits{$1/T^2$ or $1/t^2$}, this implies that the number of degrees of freedom for $u^H$, $d^H \sim \eedits{1/t^2}$, is much larger than that for $u^L$, $d^L \sim \eedits{1/T^2}$. Assuming the polynomial degree $p$ of approximation within each element is the same on both meshes, a standard \textit{a priori} finite element analysis yields the estimate
  \begin{align*}
    \left\| u - u^L \right\|_{L^2(\Omega)} &\leq C \eedits{T^{p+1}}, & \left\| u - u^H \right\|_{L^2(\Omega)} &\leq C \eedits{t^{p+1}},
  \end{align*}
for a constant $C$. Since \eedits{$t \ll T$}, and assuming the estimate above is sharp, this is an instance of the situation \eqref{eq:mf-setup}, where the unspecified norm in that equation is the $L^2$ norm on $\Omega$, and the cost function refers to computational cost. However, we note that our procedure applies to the more abstract situation \eqref{eq:mf-setup} and our procedure is not limited to a particular discretization, e.g., finite element methods.}
  
Due to the high computational cost, the number of high fidelity simulations is limited to $m$, with $m = \BigO(10)$ being a common bound.
In contrast, the low-fidelity model is much more computationally affordable with $n \gg m$ low-fidelity simulations available.

Let $\gamma = \{z_1, \ldots, z_n\}$ be a set of sample points in $D$, which define matrices
\begin{align*}
\bs{a}^L_j &= u^L\left(z_j\right), & \bs{A}^L &= \left[ \begin{array}{cccc} \bs{a}^L_1 & \bs{a}^L_2 & \ldots & \bs{a}_n^L \end{array} \right] \in \R^{d^L \times n}, \\
\bs{a}^H_j &= u^H\left(z_j\right), & \bs{A}^H &= \left[ \begin{array}{cccc} \bs{a}^H_1 & \bs{a}^H_2 & \ldots & \bs{a}_n^H \end{array} \right] \in \R^{d^H \times n}.
\end{align*}
Let $S \subset [n]$ denote a generic set of column indices, and for $\bs{A} \in \R^{d \times n}$ having columns $\bs{a}_j$ we define
\begin{align*}
S &= \left\{ j_1, \ldots, j_{|S|} \right\}, & \bs{A}_S &= \left[ \begin{array}{cccc} \bs{a}_{j_1} & \bs{a}_{j_2} & \ldots & \bs{a}_{j_{|S|}} \end{array} \right].
\end{align*}
The authors in \cite{narayan2013stochastic} showed that the structure of $\bs{A}^L$ can be used to identify a small number of column indices $S \subset [n]$, $|S| = m$, so that $\bs{A}^L_S$ can be used to form a rank-$m$ approximation to $\bs{A}^L$, and $\bs{A}^H_S$ can be used to form a rank-$m$ approximation to $\bs{A}^H$. Precisely, they form the approximations
\begin{align}\label{eq:mf-least-squares}
  \bs{A}^L &\approx \bs{A}^L_S \left(\bs{A}^L_S\right)^\dagger \bs{A}^L, & 
  \bs{A}^H &\approx \bs{A}^H_S \left(\bs{A}^L_S\right)^\dagger \bs{A}^L,
\end{align}
\edits{where $\bs{A}^\dagger$ is the Moore-Pensore pseudoinverse of $\bs{A}$.} An important observation in the above approximations is that the representation for $\bs{A}^H$ requires \textit{only} $\bs{A}^H_S$, i.e., it only requires $m = |S|$ evaluations of the high-fidelity model. The construction of $S$ is performed in a greedy fashion, precisely as the first $m$ ordered pivots in a pivoted Cholesky decomposition of $\left(\bs{A}^L\right)^T \bs{A}^L$. While $\bs{A}^H$ only represents $u^H$ on a discrete set $\gamma$, the procedure above forms the approximation
\begin{equation} \label{eqn:approx}
  u^H(z) \approx \sum_{i=1}^m c_i\edits{(z)} u^{H}(z_{j_i}),
\end{equation}
where the coefficients $c_i$ are the expansion coefficients of $u^L(z)$ in a least-squares approximation with the basis $\left\{ u^L(z_{j_i}) \right\}_{i=1}^m$. Thus, the \edits{above approximation to the} high-fidelity simulation may actually be evaluated at any location $z \in D$ if $u^L(z)$ is known. This procedure has the following advantages:
\begin{itemize}
  \item Once $m$ high-fidelity simulations have been computed and stored, the \edits{mathematical construction of the approximation} \eqref{eqn:approx} to the high-fidelity model \edits{is tantamount to evaluating the coefficients $c_i(z)$, the latter of which has computational complexity depending only on the low-fidelity model, cf. \eqref{eqn:approx}. Explicitly constructing the vector $u^H(z) \in \R^{d^H}$ naturally requires an $m$-fold addition of length-$d^H$ vectors as in \eqref{eqn:approx}.}
  \item The subset $S$ is identified via analysis of the inexpensive low-fidelity model, so that a very large set $\gamma$ may be used to properly capture the parametric variation over $D$.
  \item It is not necessary for the spatiotemporal features of the low-fidelity to mimic those of the high-fidelity model. Section 6.2 in \cite{narayan2013stochastic} reveals that $u^L$ and $u^H$ may actually be entirely disparate models, yet the approximation \eqref{eqn:approx} can be accurate. \edits{However, these are empirical observations, and a deeper understanding of precisely how disparate low- and high-fidelity models can be is an open problem.}
  \item \edits{The cost of evaluating \eqref{eqn:approx} can be split into two portions: an ``offline" stage, where the ensemble $\{ u^N(z_{j_i})\}_{i=1}^m$ compiled and stored, and an ``online" stage, where the coefficients $c_i$ must be computed for a given $z$. The offline stage requires $m$ high-fidelity simulations and $n \gg m$ low-fidelity model evaluations, and is thus expensive; the online stage requires only a single low-fidelity model evaluation, $u^L(z)$, from which the coefficients $c_i$ are computed. Thus, this procedure can be efficient when one requires an approximation to the high-fidelity model at $M \gg m$ values of $z$.}
\end{itemize}

\subsection{Subset selection}
The critical portion of the previous section's algorithm is the identification of the column subset $S$. In order to guarantee convergence of the high-fidelity approximation \edits{in an appropriate norm on $V^H$}, one must also have some correlation between the \textit{parametric} variation of the high- and low-fidelity models. \edits{Theorem 4.4 of \cite{narayan2013stochastic} gives the conditions which can guarantee quantitative proximity of the multifidelity surrogate and the high-fidelity model. These conditions are, admittedly, difficult to verify.} In this paper, we assume such parametric correlation exists and focus exclusively on the problem of identification of $S$. Abstractly, this is a problem of subset selection among the $n$ columns of $\bs{A}^L$, and thus this problem focuses entirely on the low-fidelity model. To emphasize this and to simplify notation we dispense with the $L$ and $H$ superscripts, hereafter writing
\begin{align*}
  \bs{A} &\gets \bs{A}^L, & \bs{a}_j \gets \bs{a}_j^L.
\end{align*}

The subset selection problem, identification of $S$, is an allocation problem. That is, given an accurate but expensive high-fidelity simulation, how do we allocate high-fidelity computational resources across the set of possible simulations $\gamma$? Noting the approximations \eqref{eq:mf-least-squares}, one way to accomplish this is to choose $S$ in a way that captures as much variation in the low-fidelity model as possible.

\begin{definition}\label{def:cssp}
Column subset selection problem (CSSP).
Find the column subset matrix $\mathbf{C = AS}$, where $\mathbf{C} \in \R^{d \times m}$, $\mathbf{S} \in \R^{n \times m}$ is a column selection matrix, and the $m \le n$ columns of $\mathbf{C}$ are a subset of the $n$ columns of $\mathbf{A}$, so that 
\begin{eqnarray} \label{eqn:cssp}
r = \|\mathbf{R}\|_F^2 = \|\mathbf{A - CC^{\dagger}A}\|_F^2
\end{eqnarray}
is minimized.
\end{definition}
The matrix $\bs{S}$ above mirrors the subset $S$ discussed earlier, and $\bs{C} = \bs{A}_S$. The exact CSSP is \edits{at present conjectured to be  NP-complete \cite{shitov2017column} and, if this were true, requires} the evaluation of all subsets of size $m$.  However there have been a large variety of strategies proposed to address CSSP; we summarize a selection of these strategies below:
\begin{itemize}
  \item Leverage sampling \cite{boutsidis2014near,papailiopoulos2014provable,cohen2015ridge,perry2016augmented}. The singular value decomposition (SVD) of $\mathbf{A}$ is computed or approximated, and this information is used to choose a column subset of $\bs{A}$, either randomly or deterministically. Leverage sampling has attractive theoretical guarantees, e.g., 
    $\|\mathbf{A - CC^{\dagger}A}\|_F^2 \le (1+\epsilon) \|\mathbf{A - A_k}\|_F^2$ where $\mathbf{C}$ is chosen according to the leverage score and $m$ is $\BigO(1/\epsilon)$. (The best rank-$k$ matrix $\bs{A}_k$ is defined by the Eckart-Young-Mirsky theorem.) In our results section we observe that leverage sampling on our datasets is less competitive than other sampling techniques.
\item The CUR matrix decomposition \cite{mahoney2009CUR}. This performs a column ($\mathbf{C}$) and row ($\mathbf{R}$) subset selection on  $\bs{A}$, resulting in an element-based decomposition of a matrix, so that $\mathbf{A} \approx \mathbf{CUR}$, where $\mathbf{U = C^{\dagger}AR^{\dagger}}$. The best known CUR matrix error bounds are obtained by using leverage sampling for the column and row subset selection \cite{mahoney2009CUR}, for that reason we consider our experimental comparison to leverage sampling and random sampling as sufficient to characterize related methods such as the CUR decomposition technique.
\item Deterministic interpolative decompositions (ID) \cite{cheng2005on}.  These are similar in spirit to CUR; ID methods frequently use a Gram-Schmidt-based column-pivoted QR method for subset selection, resulting in an approximation $\mathbf{A \approx CB}$, where $\mathbf{C}$ is a column subset of $\mathbf{A}$ and $\mathbf{B}$ is a coefficient matrix that minimizes the approximation error. (E.g., $\bs{B} = \bs{C}^\dagger \bs{A}$ is possible choice for $\bs{B}$.) In this paper we will use the pivoted QR method as an exemplar of deterministic ID algorithms.
\item Pivoted QR decompositions \cite{golub1965numerical,gu1996efficient}. One historical use of these routines was to select a column subset in order to reliably compute the least-squares solution of a linear system. For CSSP we run a pivoted QR routine and then sample the columns in order identified by the pivots \cite{chan1992some,papailiopoulos2014provable}. QR based subset selection has a much worse known theoretical guarantee with respect to the rank-$k$ tail bound $\|\mathbf{A - A_k}\|_F^2$ than leverage sampling, but empirically we have observed it performs much better.  Both \cite{narayan2013stochastic} and \cite{zhu2014computational} propose using a pivoted Cholesky routine in a manner that is algebraically equivalent to a pivoted QR method. Alternative linear algebraic pivoting strategies can be used as approaches to CSSP, such as partial or full pivoted LU \cite{golub2012matrix}. However, besides their mention in \cite{zhu2014computational} as a possible direction, we have not observed their use in existing CSSP work.
\end{itemize}
We note that \cite{altschuler2016greedy} showed a different type of bound for a general greedy strategy, of which the pivoted QR, LU, and Cholesky methods are specific instances.  The theoretical bound shown for the general greedy strategy is $\|\mathbf{CC^{\dagger}A}\|_F^2 \ge (1-\epsilon)\|\mathbf{DD^{\dagger}A}\|_F^2$ where $\mathbf{D}$ is the optimal subset with $\mathrm{rank}(\mathbf{D}) = k$ (defined by Definition \ref{def:cssp}), \edits{and $\mathbf{C}$ is a computed rank $m$ approximation, where $m \sim 16 k$. Table \ref{tab:analysis} gives more precise estimates.}  This bound is also of interest because large values of $\|\mathbf{CC^{\dagger}A}\|_F^2$ are related to small values of $\|\mathbf{A-CC^{\dagger}A}\|_F^2$. \edits{Working in the matrix Frobenius norm $\|\cdot\|_F$ is a common choice since this generally leads to easier analysis. The analysis in this manuscript can be used to provide estimates in the Frobenius norm; we expect estimates in other norms to be significantly more complicated, and leave investigation of other norms to future work.}

All the strategies above are described as applied to a set of Euclidean vectors, but each only requires a set of elements in a Hilbert space.
Consequently, by using the same Hilbert space assumptions from prior work \cite{narayan2013stochastic}, we are able to make use of a broad class of CSSP algorithms for the allocation portion of the multifidelity procedure under consideration.

\section{Allocation strategies for multifidelity simulation}
\edits{In this section, we start by reformulating a regularized version of the CSSP as a mixed-norm constrained least-squares problem.  This relaxation of the problem opens up to us a variety 
of different solution techniques.  We then present some theoretical justifications for why GOMP is competitive for the CSSP.  We follow this by a theoretical discussion of how various
methods we investigated compared, against each other and to GOMP, in terms of error guarantees, number of columns and asymptotic runtime to achieve those guarantees.}

\subsection{Relaxation of the exact subset problem} \label{sec:gomp}
We now reformulate a regularized version of the CSSP as a mixed-norm constrained least-squares problem. \edits{The goal of the discussion below is to write a relaxed version of the residual minimization for the CSSP as an optimization problem over a residual; such a formulation naturally suggests usage of group matching algorithms for residual minimization, where groups are matrix columns.}

Each column of the subset matrix $\mathbf{C}$ is taken from the full dataset, and therefore $\mathbf{c}_j \in \{\mathbf{a}_i| i = 1,\ldots n\}$.  
This representation can be formulated so that \edits{the residual matrix can be expressed as $\mathbf{R} = \bs{A} - \mathbf{C}\tilde{\mathbf{B}}$}, where $\tilde{\mathbf{B}} \in \R^{m \times n}$ is a coefficient matrix that essentially combines the coreset elements, or atoms, via weights.
Each data point is therefore approximated as,
\begin{equation}
\mathbf{a}_i \approx \mathbf{C}\tilde{\mathbf{b}}_i,
\end{equation}
and the residual is
\begin{align}
  r &= \|\bs{R}\|_F^2 = \sum_i \| \mathbf{a}_i - \mathbf{C}\tilde{\mathbf{b}}_i \|_2^2 = \|\mathbf{A} - \mathbf{C}\tilde{\mathbf{B}}\|_F^2 = \|\mathbf{A} - \mathbf{A}\mathbf{B}\|_F^2,
\end{align}
where $\mathbf{B} \in \R^{n \times n}$ \edits{is given by $\bs{B} = \bs{S} \tilde{\bs{B}}$. I.e., $\bs{B}$ is formed from $\tilde{\bs{B}}$ by inserting rows of zeros associated to columns of $\bs{A}$ that are not chosen in $\bs{C}$.}
The smaller matrix $\tilde{\mathbf{B}}$ is a submatrix of the larger more sparse $\mathbf{B}$, where $m$ specific (nonzero) rows are preserved. 

Instead of building $\mathbf{B}$ from $\tilde{\mathbf{B}}$, we are interested in solving for $\mathbf{B}$ directly which minimizes $r$.
The optimum of the objective $\|\mathbf{A} - \mathbf{A} \mathbf{B}\|_F^2$ would result in $\mathbf{B} = \mathbf{I}$, however if we constrain or regularize $\mathbf{B}$ so that it is row-sparse as described above, then we can compute a matrix $\mathbf{B}$ with the desired properties.
This leads to this regularized formulation for column subset selection,
\begin{equation} \label{eqn:0obj}
  \arg \min_{\mathbf{B} \in \R^{n \times n}} \|\mathbf{A - AB}\|_F^2 + \lambda \|\mathbf{B}^T\|_{0,0}
\end{equation}
where $\|\mathbf{M}\|_{0,0} = \sum_i \|\mathbf{M}_{:,i}\|_0$ is a mixed $\ell_{0,0}$ norm that induces column sparsity. 
Note the equivalence of the objective in (\ref{eqn:0obj}) to CSSP objective in (\ref{eqn:cssp}) above where $\mathbf{AB = CC^{\dagger}A}$, and the columns selected for $\mathbf{C}$ correspond to the non-zero rows of $\mathbf{B}$.
Because of the mixed $\ell_{0,0}$ norm we still have an NP-complete problem.
However in this new form, we can relax the mixed norm penalty to an $\ell_{2,1}$ norm, where $\|\mathbf{M}\|_{2,1} = \sum_i \|\mathbf{M}_{:,i}\|_2$,
\begin{equation} \label{eqn:obj}
  \mathbf{B}^* = \arg \min_{\mathbf{B} \in \R^{n \times n}} \|\mathbf{A - AB}\|_F^2 + \lambda \|\mathbf{B}^T\|_{2,1}
\end{equation}
This relaxed form also induces row sparsity of $\bs{B}$, \edits{and is more amenable to numerical methods for its solution, e.g. \cite{yuan2006model,yang2015fast}}. \edits{We have made two relaxations of the CSSP: the first is introduction of $\lambda$, and the second relaxes the mixed $\ell_{0,0}$ norm to an $\ell_{2,1}$ norm. In general, we therefore do not expect the solution to \eqref{eqn:obj} to match the solution of \eqref{eqn:0obj}, nor do we expect that it recovers the optimal CSSP solution. However, later in Theorems \ref{thm:noiseless} and \ref{thm:noisy}, we give sufficient conditions for specialized cases where this relaxation recovers the CSSP solution.}

The type of problem shown in (\ref{eqn:obj}) is known in the optimization literature as a \emph{group lasso} problems  \cite{yuan2006model}.  	
Group lasso problems have been well studied and there are a variety of algorithms available for their solution.
For the results in this paper we use group orthogonal matching pursuit (GOMP).  
Although this is a greedy algorithm and not optimal in general, it performs well for selecting subsets of columns \cite{lozano2009group}.  
Two other methods that have demonstrated efficiency in solving this type of optimization are the \emph{group least angle regression} (GLARS) and \emph{alternating direction method of multipliers} (ADMM) \cite{yuan2006model,boyd2011alternating,deng2011group}.  
We have found GOMP and GLARS to be the most useful overall for the exploration of various coreset sizes, while ADMM is well-suited to extremely large datasets \edits{\cite{deng2013group,boyd2011distributed}. The GOMP algorithm is simpler than GLARS, and since numerical tests on our particular datasets suggest no predictive advantage of one over the other, we have opted to concentrate on GOMP here. Some discussion that compares these differences between some of these group optimization problems is present in \cite{yuan2006model}, but we are unaware of quantitative results comparing how well approximate solvers for \eqref{eqn:obj} perform.}

While GOMP can be computed over a discrete data matrix $\bs{A}$, it only really makes use of the inner products of the data vectors, so that the algorithm only assumes a Hilbert space over the data points.
This means that the approach can be used in the context of continuous simulation vectors $u^L(z_i)$ as long as they exist in a proper Hilbert space. A pseudocode presentation of GOMP for column subset selection with continuous vectors is shown in Algorithm~\ref{alg:gomp}.
\edits{
The algorithm takes the input vectors, $u(\gamma)$, as well as two parameters, $\lambda$ and $\epsilon$, to control the sparsity and the precision, respectively (Algorithm~\ref{alg:gomp}, line 1).
We initialize the coefficient matrix, $\mathbf{B}$, to be the zero matrix and the active set, $\mathcal{A}$, to be empty (line 4).
Now, until \eedits{the sparsity level or the error precision} is satisfied, we loop and do the following each iteration:
(re)compute the residual matrix based on the current coefficient matrix (line 7), 
recompute the group correlations based on the residual matrix (line 8),
find the unused element with maximum group correlation and add it to the active set (lines 9 and 13), 
and recompute the coefficient matrix (line 14).
Note that the inner product matrix $\mathbf{Q}$ is computed using the appropriate inner product for the space.
The residual matrix $\mathbf{R}$ is actually the residual-correlation matrix as it represents the inner product of the residual with the input functions.
For example, considering line 7, each entry $\mathbf{R}_{i,j} = \langle u(z_i), r(z_j) \rangle = \langle u(z_i), u(z_j) - \sum_{k=1}^N b_{j,k} u(z_k) \rangle $, where $r(z_j)$ represents the residual at sample point $z_j$. 
The sparsity parameter $\lambda$ only controls the stopping criterion and is not used in the subset selection procedure. 
}

\begin{algorithm}
	\caption{\label{alg:gomp} Group orthogonal matching pursuit for subset selection of continuous vectors}
	\begin{algorithmic}[1]
		\State \textbf{Input} $u(\gamma)$ - input vectors, $\lambda$ - sparsity parameter,  $\epsilon$ - precision parameter
		\State \textbf{Output} $\mathbf{B} \in \R^{m \times p}$, coefficient matrix, $\mathcal{A}$, the subset indices.
		\Procedure{GOMP}{$u(\gamma),\lambda,\epsilon$}
			\State Set $\mathbf{Q}_{ij} = \langle u(z_i),u(z_j)\rangle$, the inner product matrix
			\State Initialize $\mathcal{A} = \emptyset$, $\mathbf{B = 0}$
                        \While{$\|\mathbf{B}^T\|_{2,1} < \frac{1}{\lambda}$}
				\State $\mathbf{R = Q - QB}$ \Comment{update residual}
				\State $\mathbf{c}_i = ||\mathbf{R}_i||_2$ 
				\State $i^* = \arg \max_i \mathbf{c}_{\mathcal{A}^{c}}$ \Comment{element with max group correlation}
                                \If{$\bs{c}_{i^*} \leq \epsilon$}
                                  \State{Break}
                                \EndIf
				\State $\mathcal{A} = \mathcal{A} \cup \{i^*\}$
				\State $\mathbf{B}_{\mathcal{A}} = \mathbf{Q}_{\mathcal{A},\mathcal{A}}^{\dagger} \mathbf{Q}_{\mathcal{A},:}$ \Comment{recompute coefficient matrix}
			\EndWhile
			\State \textbf{return} $\mathbf{B},\mathcal{A}$ 
		\EndProcedure
	\end{algorithmic}
\end{algorithm}

\subsection{Group orthogonal matching pursuit}
We now present some theoretical justification for why GOMP is competitive for the CSSP. Our two results are Theorem \ref{thm:noiseless} and \ref{thm:noisy}. In particular, Theorem \ref{thm:noisy} shows that GOMP can identify a small representative column subset under some assumptions on the matrix and the representative set.

Let $\bs{A} \in \R^{d \times n}$ be a matrix whose $n$ columns consist of snapshot or feature vectors. 
We assume that the columns of $\bs{A}$ have unit norm, i.e., if $\bs{a}_j$ is the $j$th column of $\bs{A}$, then we assume that 
\begin{align}\label{eq:normalization}
  \left\| \bs{a}_j \right\|_2 = 1.
\end{align}
Our goal is to construct a low-rank approximation to every feature (column) in $\bs{A}$ using linear combinations of a (small) subset of feature vectors in $\bs{A}$. 
As described in the previous section, our strategy for accomplishing this is to solve the following optimization problem for the matrix $\bs{B} \in \R^{n \times n}$:
\begin{align}\label{eq:objective}
  \min \left\| \bs{A} - \bs{A} \bs{B} \right\|_F^2 + \lambda \left\| \bs{B}^T \right\|_{2,1},
\end{align}
where 
\begin{align*}
  \left\| \bs{B}^T \right\|_{2,1} = \sum_{j=1}^n \left\|\bs{b}_{j,:} \right\|_2,
\end{align*}
and $\bs{b}_{j,:}$ the $j$th row of $\bs{B}$. The mixed norm $\|\cdot\|_{2,1}$ promotes column sparsity of its argument. 
In the context of the product $\bs{A} \bs{B}$, column sparsity of $\bs{B}^T$ implies that a small number of columns of $\bs{A}$ are used to approximate the remaining columns.

Because we solve \eqref{eq:objective} using a group orthogonal matching pursuit (GOMP) algorithm, we proceed to rewrite this problem in the language of GOMP algorithms. Let $\bs{a}_i$, $i \in [n]$, denote the $i$th column of $\bs{A}$. 
We call a set of (column) indices $S \subset [n]$ a \textit{basis} set for $\bs{A}$ if 
\begin{align*}
  \mathrm{span} \left\{ \bs{a}_i \; \big| \; i \in S \right\} = \mathrm{range}\left(\bs{A}\right).
\end{align*}
Basis sets are not unique in general.
A \edits{\textit{complement}} set of indices equals $[n]\backslash S$ for a given basis set $S$. 
Our first result is a consistency result: given a basis set that satisfies reasonable assumptions, GOMP can identify this set.
\begin{theorem}\label{thm:noiseless}
  Let $\bs{A} \in \R^{d \times n}$ have columns $\bs{a}_i \in \R^d$, and let $S_{\mathrm{g}} \subset [n]$ be a basis set for $\bs{A}$. For each $j \in [n]$, there is a unique expansion of $\bs{a}_j$ in columns of the basis set:
  \begin{subequations}\label{eq:noiseless-conditions}
  \begin{align}\label{eq:C-def}
    \bs{a}_j = \sum_{i \in S_{\mathrm{g}}} D_{i,j} \bs{a}_i.
  \end{align}
  \edits{Use the elements $D_{i,j}$ to define the matrix $\bs{D}$. If both conditions}
  \begin{align}\label{eq:A-consistency}
    \bar{D} \coloneqq \max_{j \in [n] \backslash S_{\mathrm{g}}} \sum_{i \in S_{\mathrm{g}}} |D_{i,j}| < 1, \hskip 25pt \edits{\epsilon > 0,}
  \end{align}
  \end{subequations}
  \editsbrittle{are satisfied, then GOMP, i.e., Algorithm \ref{alg:gomp} terminates after at most $|S_{\romg}|$ steps. At termination, the algorithm identifies a subset of $S_{\mathrm{g}}$ if fewer than $|S_{\romg}|$ steps are taken, and identifies $S_{\romg}$ if $|S_{\romg}|$ steps are taken.}
\end{theorem}
\begin{proof}
For any matrices $\bs{A}$ and $\bs{B}$ of appropriate size, the following identity holds:
\begin{align*}
  \mathrm{vec} \left( \bs{B}^T \bs{A}^T \right) = \left( \bs{A} \otimes \bs{I}_n \right) \mathrm{vec}\left(\bs{B}^T\right),
\end{align*}
where $\otimes$ is the Kronecker product, and $\mathrm{vec}\left(\cdot\right)$ is the vectorization (vertical concatenation of columns) of its argument.  
Defining $\bs{g} \coloneqq \mathrm{vec}\left(\bs{B}^T\right)$, $\bs{f} \coloneqq \mathrm{vec}\left(\bs{A}^T\right)$, and $\bs{H} = \left( \bs{A} \otimes \bs{I}_n \right)$, we can write our desired relation $\bs{A} \approx \bs{A} \bs{B}$ in vectorized format:
\begin{align}\label{eq:gomp-linear-system}
  \bs{H} \bs{g} \approx \bs{f}.
\end{align}
The goal of \eqref{eq:objective}, aiming to promote row sparsity of $\bs{B}$, is now translated into a group sparsity of entries of $\bs{g} = \mathrm{vec}\left(\bs{B}^T\right)$. 
Each group of elements in $\bs{g}$ corresponds to an individual row of $\bs{B}$. We seek to prove that GOMP applied to this system recovers the basis groups. \edits{Our main tool to accomplish this is Corollary 1 in \cite{lozano2009group}; this corollary shows that a version of GOMP in that reference identifies the correct groups. We will first show this for our case, and then explain how this can be used to conclude the identification of correct groups for Algorithm \ref{alg:gomp}.}

We will let $G$ denote a generic subset of $[n^2]$. We define the groups $G_i \subset [n^2]$, $i = 1, \ldots, n$ as sets of $n$ sequential indices,
\begin{align*}
  G_i = \left\{ n(i-1) + 1, n(i-1) + 2, \ldots, n i \right\},
\end{align*}
and thus $G_1, \ldots, G_n$ partition $[n^2]$. 
We denote by $\bs{H}_{G}$ the matrix $\bs{H}$ restricted to the $G$-indexed columns. \editsbrittle{(Recall that $G$ is a subset of $[n^2]$ and $H$ has $n^2$ columns.)} The basis columns $S_{\mathrm{g}}$ can be translated into a set of basis groups $G_{\mathrm{g}}$:
\begin{align*}
  S_{\mathrm{g}} \subset [n] &\Longrightarrow G_{\mathrm{g}} \subset [n^2], \hskip 10pt G_{\mathrm{g}} = \cup_{i \in S_{\mathrm{g}}} G_i
\end{align*}
The complement of $S_{\mathrm{g}}$ in $[n]$, and the associated set $G_{\mathrm{d}}$, are defined accordingly
\begin{align*}
  S_{\romd} &\coloneqq [n] \backslash S_{\mathrm{g}}, & G_{\romd} &= \cup_{i \in S_{\romd}} G_i.
\end{align*}
We use $\bs{A}_S$ to denote selection of the $S$-indexed columns of $\bs{A}$, and $\bs{H}_G$ to denote selection of the $G$-indexed columns of $\bs{H}$. Under the assumption \eqref{eq:normalization}, we have $\bs{H}_{G_j}^T \bs{H}_{G_j} = \bs{I}_n$ for each group $j=1, \ldots, n$.

Row sparsity of $\bs{B}$ means that we wish to choose a small (sparse) number of groups $G_j$. Our next step is to show that the GOMP algorithm in \cite{lozano2009group} applied to \eqref{eq:gomp-linear-system} identifies groups of indices associated with the column indices $S_{\mathrm{g}}$. Consider the matrix,
\begin{align*}
  \bs{M} = \bs{A}_{S_{\mathrm{g}}}^T \bs{A}_{S_{\mathrm{g}}}.
\end{align*}
Since $S_{\mathrm{g}}$ is a set of basis indices, then $\bs{M}$ is invertible. This implies
\begin{align*}
  \bs{H}_{G_{\mathrm{g}}}^\dagger = \left( \bs{A}_{S_{\mathrm{g}}} \otimes \bs{I}_n \right)^\dagger = \bs{A}_{S_{\mathrm{g}}}^\dagger \otimes \bs{I}_n = \left(\bs{M}^{-1} \bs{A}_{S_{\mathrm{g}}}^T\right) \otimes \bs{I}_n
\end{align*}
Then 
\begin{align*}
  \bs{H}_{G_{\mathrm{g}}}^\dagger \bs{H}_{G_{\romd}} =  \left( \left(\bs{M}^{-1} \bs{A}_{S_{\mathrm{g}}}^T\right) \otimes \bs{I}_n \right) \left( \bs{A}_{S_{\romd}} \otimes \bs{I}_n \right) = \left( \bs{M}^{-1} \bs{A}_{S_{\mathrm{g}}}^T \bs{A}_{S_{\romd}}\right) \otimes \bs{I}_n.
\end{align*}
Define $\bs{D} \coloneqq \bs{M}^{-1} \bs{A}_{S_{\mathrm{g}}}^T \bs{A}_{S_{\romd}} \in \R^{s \times (n-s)}$, and note that it is comprised of the entries $D_{i,j}$ defined in \eqref{eq:C-def}. We need to introduce a mixed vector norm. Let $q \in \N$ be arbitrary; for a vector $\bs{x} \in \R^{q n}$, we define
\begin{align*}
\bs{x} = \left( \begin{array}{c} \bs{x}_1 \\ \bs{x}_2 \\ \vdots \\ \bs{x}_q \end{array} \right) \in \R^{q n}, \hskip 10pt \bs{x}_j \in \R^n \hskip 15pt \Longrightarrow \hskip 15pt \left\| \bs{x} \right\|_{2,1} = \sum_{j=1}^q \left\| \bs{x}_j \right\|_2.
\end{align*}
Given such a vector $\bs{x}$, we use notation $\bs{x}_j$ to mean a length-$n$ vector from the elements of $\bs{x}$ as above. We have
\begin{align*}
  \left\| \left( \bs{D} \otimes \bs{I}_n \right) \bs{v} \right\|_{2,1} &= \left\| \sum_{j=1}^{n-s} \left( \bs{d}_j \otimes \bs{v}_j \right) \right\|_{2,1} \leq \sum_{j=1}^{n-s} \left\| \bs{d}_j \otimes \bs{v}_j \right\|_{2,1} \\
                                                                       &= \sum_{j=1}^{n-s} \sum_{i=1}^s \left| D_{i,j} \right| \|\bs{v}_j\|_2  \leq \left( \max_{1 \leq j \leq n-s} \sum_{i=1}^s \left| D_{i,j} \right| \right) \sum_{j=1}^{n-s} \left\| \bs{v}_j \right\|_2 = \left\| \bs{D} \right\|_1 \| \bs{v}\|_{2,1}
\end{align*}
Chaining together these results with the assumption \eqref{eq:A-consistency}, we have shown
\begin{align*}
  \left\| \bs{H}^+_{G_{\mathrm{g}}} \bs{H}_{G_{\romd}} \right\|_{\ast 2,1} \coloneqq \sup_{\bs{v} \in \R^{(n-s)n}\backslash \{\bs{0}\} } \frac{\left\| \left(\bs{D} \otimes \bs{I}_n \right) \bs{v} \right\|_{2,1}}{\|\bs{v}\|_{2,1}} < 1, 
\end{align*}
\edits{where we have used $\|\cdot\|_{\ast 2,1}$ to denote the norm induced by the $\|\cdot\|_{2,1}$ norm on vectors, and not the mixed $\ell_{2,1}$ elementwise norm on matrices. Under this condition, Corollary 1 in \cite{lozano2009group} shows that the GOMP algorithm in that reference applied to $\bs{H} \bs{g} \approx \bs{f}$ identifies the basis groups in $\bs{d}$; equivalently, that GOMP algorithm identifies the basis columns of $\bs{A}$. This is \textit{almost} the result we require: our GOMP method, Algorithm \ref{alg:gomp}, is almost identical to the algorithm in \cite{lozano2009group}, the only difference being the presence of an additional regularization parameter $\lambda$ that influences the stopping criterion. Without this criterion, then Algorithm \ref{alg:gomp} coincides with the algorithm in \cite{lozano2009group}. Therefore, Algorithm \ref{alg:gomp} selects the same groups, and differs only in the termination.

  Since the algorithm in \cite{lozano2009group} identifies $S_{\romg}$ under our assumptions, then Algorithm \ref{alg:gomp} also identifies this set if it takes $|S_{\romg}|$ steps. Furthermore, Algorithm \ref{alg:gomp} takes at most $|S_{\romg}|$ steps, since at step $|S_{\romg}|+1$, the residual vanishes and the while loop terminates. Finally, the fact that Algorithm \ref{alg:gomp} identifies $S_{\romg}$ after $|S_{\romg}|$ steps implies that, if it terminates before $|S_{\romg}|$ steps due to the $\lambda$ stopping criterion, then the identified set at termination must be a subset of $S_{\romg}$.}
\begin{align*}
\end{align*}
\vskip-30pt
\end{proof}

The condition \eqref{eq:A-consistency} on the $D_{i,j}$ is a nontrivial requirement, \edits{both on the matrix $\bs{A}$, and on a basis set $S_{\mathrm{g}}$}, and it is difficult to verify in practice. \edits{This restriction is essentially an identification of classes of matrices for which we can guarantee that GOMP can be effective.} In addition, since the set $S_{\mathrm{g}}$ defining a basis group is not unique for a given matrix, many choices of $S_{\mathrm{g}}$ may not satisfy \eqref{eq:A-consistency}. One significant disadvantage of this analysis is that the condition \eqref{eq:A-consistency} cannot be checked before the set $S_{\mathrm{g}}$ has been identified.

For $\bs{A} \in \R^{d \times n}$, assuming $d \leq n$ and that $\bs{A}$ has full rank, we require $|S_{\mathrm{g}}| = d$ in order to satisfy the relation \eqref{eq:C-def}. However, $d$ can be very large and so in practice we will identify a set of features $S$ with size $|S| < d$; this suggests that \eqref{eq:C-def} can only be satisfied approximately in this case. Following the analysis in \cite{lozano2009group}, we can provide \edits{a} robust version of the above theorem, showing success of GOMP when the relation \eqref{eq:C-def} is satisfied only approximately.

We require a little more notation to proceed: With $\bs{A}$ and $S_{\mathrm{g}}$ as in Theorem \ref{thm:noiseless}, let $\bs{A}_{S_{\mathrm{g}}}$ be as defined in the proof of that theorem, i.e., the submatrix of $\bs{A}$ formed from the columns indexed by $S_{\mathrm{g}}$. Now define the smallest eigenvalue of $\bs{A}^T_{S_{\mathrm{g}}} \bs{A}_{S_{\mathrm{g}}}$:
\begin{align}\label{eq:lambda-min}
  \bar{\lambda} \coloneqq \lambda\left( \bs{A}^T_{S_{\mathrm{g}}} \bs{A}_{S_{\mathrm{g}}} \right) = \min_{\bs{v} \in \R^{|S_{\mathrm{g}}|} \backslash \{\bs{0}\}} \frac{ \left\| \bs{A}^T_{S_{\mathrm{g}}} \bs{A}_{S_{\mathrm{g}}} \bs{v}\right\|_2}{ \left\| \bs{v} \right\|_2} > 0
\end{align}
where the inequality is true since the columns of $\bs{A}_{S_{\mathrm{g}}}$ are a basis for $\mathrm{range}(\bs{A})$. 
\begin{theorem}\label{thm:noisy}
  Let $\bs{A} \in \R^{d \times n}$ have normalized columns $\bs{a}_i \in \R^d$, and let $S_{\mathrm{g}} \subset [n]$ be a basis set for $\bs{A}$ such that the columns of $\bs{A}$ satisfy
  \begin{subequations}\label{eq:noisy-conditions}
  \begin{align}\label{eq:C-def-noisy}
    \bs{a}_j = \sum_{i \in S_{\mathrm{g}}} D_{i,j} \bs{a}_i + \bs{n}_j,
  \end{align}
  and assume that each $\bs{n}_j \in \R^d$ for $j= [n]\backslash S_{\mathrm{g}}$ has independent and identically distributed components, each a normal random variable with mean $0$ and variance $\sigma^2$.  For any $\eta \in (0, 1/2)$, let the stopping tolerance $\varepsilon$ in the GOMP Algorithm \ref{alg:gomp} satisfy
  \begin{align}\label{eq:epsilon-noisy}
    \varepsilon > \frac{\sigma\sqrt{2 \cdot n \cdot d \log(2 \cdot  n \cdot d/\eta)}}{1 - \bar{D}}.
  \end{align}
  Assume that the $D_{i,j}$ satisfy \eqref{eq:A-consistency}, and further that they satisfy
  \begin{align}\label{eq:C-noisy}
    \min_{i \in S_{\mathrm{g}}} \sqrt{\sum_{j=1}^n D_{i,j}^2} > \frac{\varepsilon \sqrt{8}}{\bar{\lambda}}.
  \end{align}
  \end{subequations}
  Then \edits{with probability at least $1 - 2 \eta$}, GOMP identifies the set $S_{\mathrm{g}}$ and computes a solution $\bs{B}$ satisfying
  \begin{align*}
    \max_{i,j} \left| B_{i,j} - D_{i,j} \right| \leq \sigma \sqrt{\frac{2 \log( \edits{2} |S_{\mathrm{g}}|/\eta)}{\bar{\lambda}}}.
  \end{align*}
\end{theorem}
\begin{proof}
  Let $\bs{D}$ be the matrix formed from the elements $D_{i,j}$, which is row-sparse, having non-zero entries only when $i \in S_{\mathrm{g}}$.  As before, we seek a computational approximation $\bs{B}$ to the (unknown) $\bs{D}$. The model \eqref{eq:C-def-noisy} results in the following linear system to determine $\bs{B}$,
  \begin{align*}
    \edits{\bs{A}\bs{B} \approx \bs{A} = \bs{A} \bs{D} + \bs{N}},
  \end{align*}
  where $\bs{N}$ is a concatenation of the columns $\bs{n}_j$. (We define $\bs{n}_i = \bs{0}$ for $i \in S_{\mathrm{g}}$.) Following the proof of Theorem \ref{thm:noiseless}, the vectorization of the transpose of the above equation is
  \begin{align*}
    \bs{H} \bs{g} = \bs{f} + \bs{n}
  \end{align*}
  where $\bs{n} = \mathrm{vec}\left(\bs{N}^T\right)$ and $\bs{f} = \mathrm{vec}\left(\bs{D}^T \bs{A}^T\right)$. The above equation holds when $\bs{g}$ has group sparsity defined by the set $S_{\mathrm{g}}$. We seek to show that GOMP applied to this system, having knowledge only of $\bs{H}$ and $\bs{f} + \bs{n}$, can identify groups of $\bs{g}$ corresponding the column indices $S_{\mathrm{g}}$ for $\bs{A}$. 

  The system above is simply a ``noisy" version of GOMP, where $\bs{n}$ is a noise vector. We seek to apply Theorem 3 in \cite{lozano2009group} to show the conclusions. Assumptions \eqref{eq:noisy-conditions} and \eqref{eq:A-consistency} are among the stipulations of this theorem. The remaining stipulation is 
  \begin{align}\label{eq:thm-noisy-temp}
    \lambda_{\mathrm{min}} \left( \bs{H}^T_{G_{\mathrm{g}}} \bs{H}_{G_{\mathrm{g}}} \right) \leq 1,
  \end{align}
  where $\lambda_{\mathrm{min}} \left(\cdot\right)$ denotes the minimium eigenvalue of the symmetric matrix input. We can show this via applications of Kronecker product properties. First we have
  \begin{align*}
    \bs{H}_{G_{\mathrm{g}}}^T \bs{H}_{G_{\mathrm{g}}} = \left( \bs{A}_{S_{\mathrm{g}}} \otimes \bs{I}_n \right)^T \left( \bs{A}_{S_{\mathrm{g}}} \otimes \bs{I}_n \right) = \left(\bs{A}_{S_{\mathrm{g}}}^T \bs{A}_{S_{\mathrm{g}}}\right) \otimes \left(\bs{I}_n \bs{I}_n \right).
  \end{align*}
  If $\lambda_i$, $i=1, \ldots, |S_{\mathrm{g}}|$ are the eigenvalues of $\bs{A}_{S_{\mathrm{g}}}^T \bs{A}_{S_{\mathrm{g}}}$, then they are also the eigenvalues of $\left(\bs{A}_{S_{\mathrm{g}}}^T \bs{A}_{S_{\mathrm{g}}}\right) \otimes \left(\bs{I}_n \bs{I}_n \right)$, each having multiplicity $n$. Thus, 
  \begin{align*}
    \lambda_{\mathrm{min}} \left( \bs{H}^T_{G_{\mathrm{g}}} \bs{H}_{G_{\mathrm{g}}} \right) = \lambda_{\mathrm{min}} \left( \bs{A}^T_{S_{\mathrm{g}}} \bs{A}_{S_{\mathrm{g}}} \right) = \bar{\lambda},
  \end{align*}
  with $\bar{\lambda}$ defined in \eqref{eq:lambda-min}. The matrix $\bs{A}_{S_{\romg}}^T \bs{A}_{S_{\romg}}$ is positive-definite and has diagonal entries equal to 1 since the columns $\bs{a}_i$ are normalized. Thus,
  \begin{align*}
    |S_{\romg}| = \mathrm{trace} \left(\bs{A}_{S_{\romg}}^T \bs{A}_{S_{\romg}} \right) = \sum_{i=1}^{|S_{\romg}|} \lambda_i \geq |S_{\romg}| \bar{\lambda},
  \end{align*}
  so that $\bar{\lambda} \leq 1$, showing \eqref{eq:thm-noisy-temp}. Having shown this we may apply Theorem 3 in \cite{lozano2009group}, which yields the conclusion of the theorem.
\end{proof}
Note that \eqref{eq:C-def-noisy} can be satisfied with $|S_{\romg}| < d$ even if $\bs{A}$ has full rank. When the columns of $\bs{A}$ are feature vectors that are outputs of scientific models, then the assumption that columns are perturbed by random noise is not realistic. However, the main purpose of Theorem \ref{thm:noisy} is to show that GOMP is robust to relaxations of \eqref{eq:C-def}.

\subsection{Theoretical discussion}
We \eedits{describe briefly} how the various methods we have investigated compare in terms of error \eedits{or projection norm} guarantees, number of columns and asymptotic runtime to achieve those guarantees. \eedits{We present certain results that we believe are relevant; the results we review are not meant to be comprehensive.}

In Table~\ref{tab:analysis} the relationship between $\epsilon$ and $m$ for deterministic leverage score sampling assumes a power-law distribution, parameterized as $\ell_i = \frac{1}{i^{1+\eta}}$, on the leverage scores.
\edits{The algorithms Cholesky, QR, LU, and GOMP are all greedy algorithms, and thus be analyzed using results from \cite{altschuler2016greedy}, which analyzes a general \emph{greedy} column subset algorithm.}
For QR we report the multiplicative error reported in \cite{gu1996efficient}, although the results from \cite{altschuler2016greedy} apply as well.

We observe \eedits{among the results we compile} that leverage sampling has the most \eedits{attractive} theoretical error bounds, because the bound is multiplicative and depends on the level of error, $\epsilon$, which is \eedits{a satisfactory metric} in applications.
The rest of the methods have the same error bound from \cite{altschuler2016greedy}, which is not directly related to the best rank-k approximation.
In addition QR has a multiplicative relationship to the best rank-k approximation, but the coefficient grows exponentially with respect to $k$. 

\eedits{In summary, the various algorithms we list in Table \ref{tab:analysis} can in theory be effective in choosing subsets. While the theoretical projection and error bounds do not allow us to conclude in practice which algorithm is superior, we aim to demonstrate in the results section that, for a variety of multifidelity problems, the GOMP strategy appears to be more effective than the rest.}

\begin{table}
\resizebox{\textwidth}{!}{%
\begin{tabular}{l|l|l|l}
\hline
Algorithm & Runtime & Theoretical error & \eedits{$m = m(k,\epsilon)$} \\
\hline
Leverage \cite{papailiopoulos2014provable} & $\BigO(n^3)$ & $\|\mathbf{A - CC^{\dagger}A}\|_F^2 \le (1+\epsilon) \|\mathbf{A - A_k}\|_F^2$ & $\eedits{m}=\max\left(\left(\frac{2k}{\epsilon}\right)^{\frac{1}{1+\eta}},\left( \frac{2k}{\eta \epsilon} \right)^{\frac{1}{\eta}}, k \right)$ \\
Cholesky \cite{altschuler2016greedy} & $\BigO(n^2m)$ &  $\|\mathbf{CC^{\dagger}A}\|_F^2 \ge (1-\epsilon)\|\mathbf{DD^{\dagger}A}\|_F^2$ & $m = \frac{16k}{\epsilon \sigma_{\min}(DD^{\dagger})}$ \\
QR \cite{gu1996efficient} & $\BigO(n^2m)$ &  $\|\mathbf{A - CC}^{\dagger}\mathbf{A}\|_2 \le \left( 1 + \sqrt{n - k} \cdot 2^k \right) \|\mathbf{A - A_k}\|_2$ & $m = k$ \\
LU \cite{altschuler2016greedy} & $\BigO(n^2m)$ &  $\|\mathbf{CC^{\dagger}A}\|_F^2 \ge (1-\epsilon)\|\mathbf{DD^{\dagger}A}\|_F^2$ & $m = \frac{16k}{\epsilon \sigma_{\min}(DD^{\dagger})}$ \\
GOMP \cite{altschuler2016greedy} & $\BigO(n^2m^2)$ &  $\|\mathbf{CC^{\dagger}A}\|_F^2 \ge (1-\epsilon)\|\mathbf{DD^{\dagger}A}\|_F^2$ & $m = \frac{16k}{\epsilon \sigma_{\min}(DD^{\dagger})}$ \\
\hline
\end{tabular}%
}
\caption{\textbf{Theoretical comparison}
among the subset-selection methods discussed here. \eedits{Each method computes a size-$m$ subset that can achieve a multiplicative-type bound relative to the optimal size-$k$ subset, where $m$ depends on $k$ and an error tolerance $\epsilon$. $\mathbf{C}$ is the computed subset matrix of rank $m$ using the indicated algorithm, and $\mathbf{D}$ is the optimal subset matrix of rank $k$, where optimal means with respect to the projection error committed in the norm used in the ``Theoretical error" column. $\epsilon$ is an error tolerance, and $\eta$ is a weighting parameter for leverage scores.}
}\label{tab:analysis}
\end{table}

\section{Experimental results}

\subsection{Methods compared}
We use a collection of methods from the literature both from previous work in multifidelity simulation approximation and in the more general CSSP literature.
Previous work in multifidelity simulation approximation has primarily used pivoted Cholesky (chol) for subset selection, although authors have alluded to the possibility of using other pivoted decomposition methods, such as partially-pivoted LU decomposition (lu) and pivoted QR decomposition (qr).
All of these decomposition methods have a general approach of selecting the element with the largest residual norm after removing contribution from previous selected elements, with the exception of LU which does this but using the element with the maximum residual element magnitude (rather than full norm).
From the CSSP literature we compare against leverage score sample (lev) that makes use of singular value decomposition (SVD) of the Gram matrix to make the selection.
Finally we compare against the proposed method, grouped orthogonal matching pursuit (gomp), which was explained in detail in Section~\ref{sec:gomp}.
We also include results using a uniformly random selection (rand) and a best rank-k approximation (rank-k) as contextual benchmarks. \edits{Here, rank-k means the best approximating matrix of rank $k$ measured in the $\ell^2$ or Frobenius norm. The Eckart-Young-Mirsky theorem gives an explicit construction of this matrix as a truncated SVD of the full matrix.}
All methods used are summarized in Table~\ref{tab:methods}.

\begin{table}
\centering
\begin{tabular}{l|l}
\hline
Abbreviation & Algorithm \\
\hline
rand & Uniform random \\
lev & Deterministic leverage score \cite{papailiopoulos2014provable} \\
qr & Pivoted QR \cite{gu1996efficient} \\
chol & Pivoted Cholesky \cite{narayan2013stochastic} \\
lu & Partially pivoted LU \cite{golub2012matrix} \\
gomp & Group orthogonal matching pursuit (Section~\ref{sec:gomp})\\
rank-k & Best rank-k approximation \cite{golub2012matrix} \\
\hline
\end{tabular}
\caption{%
Summary of methods used in the empirical comparison for choosing for resource allocation.
}
\label{tab:methods}
\end{table}

We report error as a squared $\ell_2$ error normalized by the sum of squared $\ell_2$ norms of the original data, $E = \frac{\sum_{i}\|\mathbf{X}_i - \tilde{\mathbf{X}}_i\|_2^2}{\sum_{i}\|\mathbf{X}_i\|_2^2}$, where $\mathbf{X}_i$ is the $i$-th simulation and $\tilde{\mathbf{X}}_i$ is the proxy estimation for the $i$-th simulation.

\subsection{Burgers equation}
We use the viscous Burgers equation to evaluate the method in a well-studied and familiar setting.
We introduce uncertain perturbations to the left boundary condition and the viscosity parameter, similar to the setting studied in \cite{zhu2014computational}.

Specifically we use,
\begin{eqnarray}
&u_t + uu_x = \nu u_{xx}, &x \in (-1,1),\\
&u(-1) = 1 + \delta, &u(1) = -1
\end{eqnarray}
where $\nu \sim U(0.1,1)$ and $\delta \sim U(0,0.1)$ are both uniform random variables in the respective ranges.
This same problem with only the boundary perturbation (fixed viscosity) was studied in \cite{zhu2014computational} because it is extremely sensitive to boundary perturbations \cite{xiu2004supersensitivity}.
We found that adding a second element of uncertainty around the viscosity parameter, $\nu$, makes the problem more interesting from the perspective of attempting a linear approximation.

\subsubsection{Setup}
We sampled this parameter space using $400$ total samples by taking $20$ uniform grid samples within the boundary perturbation range, $(0,0.1)$, and $20$ uniform grid samples within the viscosity perturbation range $(0.1,1.0)$.
For each parameter sample pair, we ran the low-fidelity and high-fidelity simulation for a set number of $t=1.5 \times 10^{6}$ simulation steps, resulting in $400$ different simulation results for the low-fidelity model and another $400$ different simulation results for the high-fidelity model.

\subsubsection{Analysis}
The ultimate goal is to understand how each subset selection method performs using the low-fidelity for selection, while the error is computed with respect to the high-fidelity estimation.
To test this, we use each method to select a subset and reconstruction weights using the low-fidelity dataset, and then test the reconstruction error in both the low-fidelity samples and the high-fidelity samples.

We observed that all methods improved with each additional sample available, but certain methods do better overall and some outperform all others when the number of samples is extremely small.
This latter case is important for our setting because we are in the setting where only a few samples can be computed in the high-fidelity model.
We found that overall, deterministic leverage-score sampling performed the worst, even after a relatively large number of samples have been found. \edits{Note, however, that leverage sampling is frequently used in the context of very large data sets, and in this particular experiment we have access to only 400 samples. Thus, leverage sampling may be a good strategy for column subset selection if more data were available, or when a larger number of subsamples are selected.
In this experiment, random sampling appears to do better on average than deterministic leverage score sampling. 
However when we look at the worst-case example from a random subset, we see that all other methods do better for small subsets.  
This last point is important when considering an allocation strategy of only a few high-cost simulations, because while on average random does well it can do quite poorly in specific cases.
}
The LU-based partial pivoting sampling approach did better than random, but worse overall than the remaining methods.
In all of our experiments, QR and Cholesky based sampling performed nearly identically.  This can be understood as follows:  when acting on a symmetric positive-definite matrix like the Gramian, the QR and Cholesky pivoting strategies become nearly equivalent.
Finally, the proposed GOMP-based sampling approach performs the best overall, and particularly for very small subset setting.
These results are summarized in Figure~\ref{fig:burgers:projerr} which shows the reconstruction error in the low-fidelity model and the high-fidelity model against the exact solution.

\edits{The right-hand panel of Figure \ref{fig:burgers:projerr} (the high-fidelity error plot) only showcases errors from this multifidelity procedure, and not from, e.g., a best rank-$k$ approximation on the high-fidelity ensemble of solutions. Our focus here is on the empirical study of the quality of high-fidelity subset selection using the low-fidelity data; under this focus, the shown plots are the most relevant and error study of a high-fidelity rank-$k$ approximation falls out of scope. Instead of subset selection quality, one can ask about how using the low-fidelity model in the reconstruction procedure affects the accuracy. I.e., this is a question about how the approximation to $\bs{A}^H$ in \eqref{eq:mf-least-squares} would compare with an approximation using ``high-fidelity" coefficients. I.e., it is a question concerning the accuracy of
  \begin{align*}
    \bs{A}^H \approx \bs{A}^H_S \left(\bs{A}^H_S\right)^\dagger \bs{A}^H, \hskip 15pt \textrm{ vs } \hskip 15pt \bs{A}^H \approx \bs{A}^H_S \left(\bs{A}^L_S\right)^\dagger \bs{A}^L,
  \end{align*}
  once $S$ has been determined. The effect of this approximation is not the central goal of this paper, and so we leave this question for future investigation. However, we note that this question has been investigated in some recent work. \cite{hampton_parametric/stochastic_2017,kesh17}. The remaining figures in this manuscript therefore focus mainly on studies pertaining to quality of the subset selection.
}

\begin{figure}[tb]
 \centering 
 \includegraphics[width=.49\textwidth]{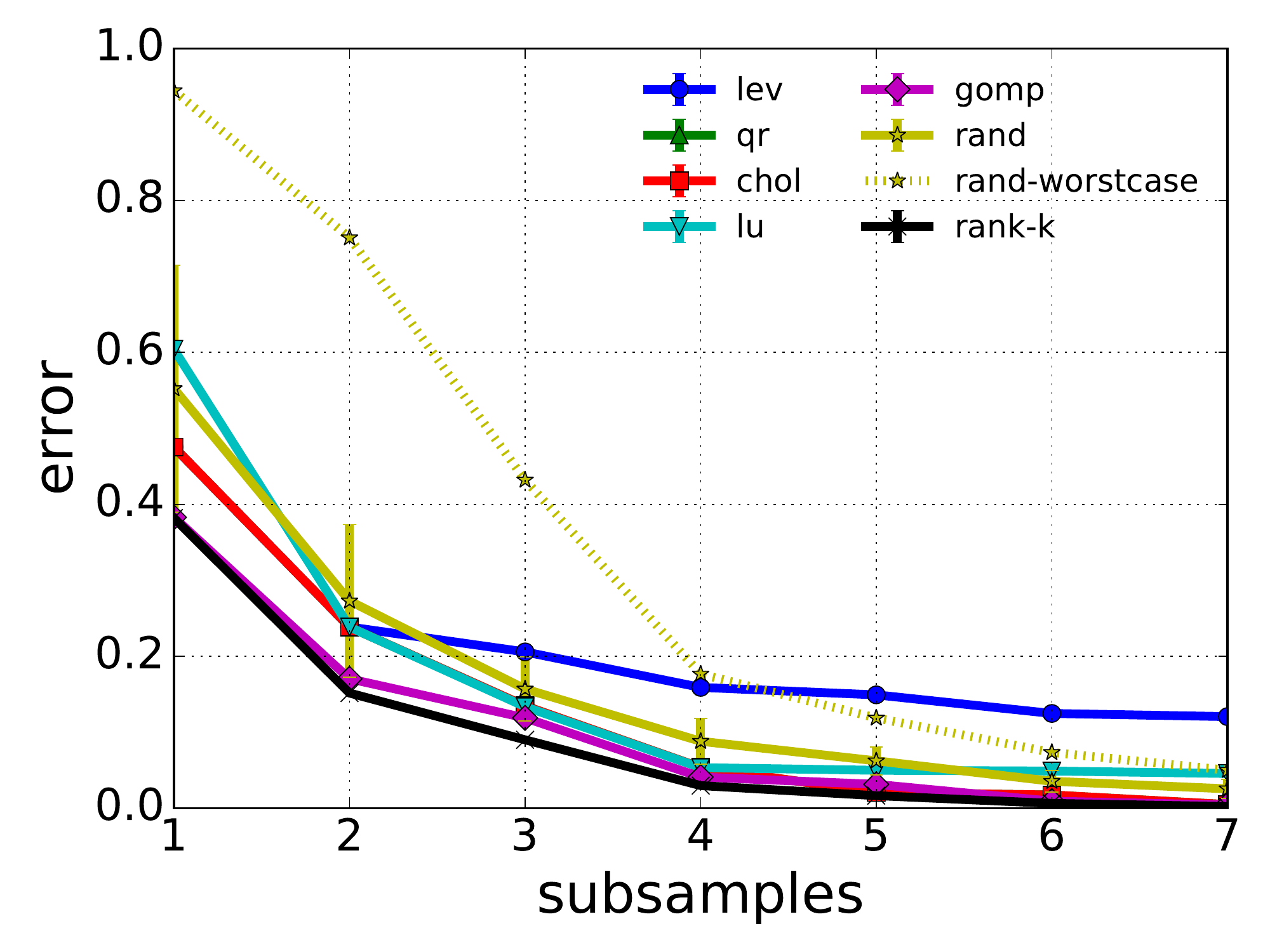} 
 \includegraphics[width=.49\textwidth]{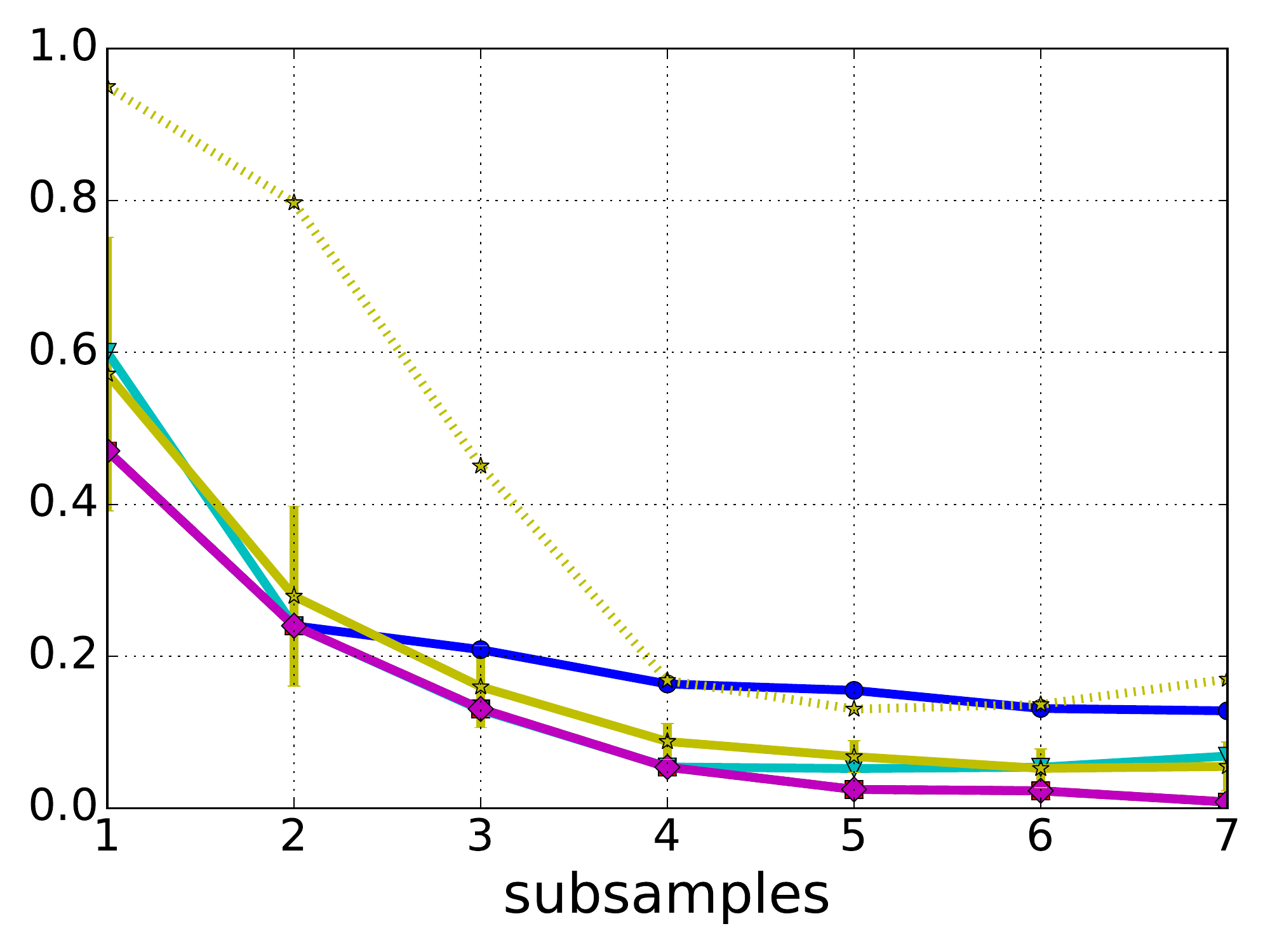} \\
\phantom{A} \hspace{.8in} \textbf{Low-fidelity model} \hfill \textbf{High-fidelity model} \hspace{.8in}
 \caption{
	\textbf{Burgers equation:} Reconstruction error among subset methods for solutions to Burgers equation.
	\emph{Left:} the reconstruction error of the rest of the low-fidelity simulation dataset when using the indicated subset chosen using various methods.  
	\emph{Right:} the reconstruction error but for the high-fidelity simulations with the subset chosen using the low-fidelity samples.
}
\label{fig:burgers:projerr}
\end{figure}

\subsection{Double pendulum}
We consider a classic double pendulum problem.  
A similar problem was also considered in \cite{narayan2013stochastic}, and we use the same setup and assumptions here.
The problem is parameterized by the two pendulum angles, $\theta_1, \theta_2$, the lengths of the pendulums, $\ell_1, \ell_2$, and the mass of each pendulum, $m_1,m_2$, as well as the gravity coefficient $g$.
The high fidelity model corresponds to the solution of Equation (6.5) in \cite{narayan2013stochastic} parameterized by $(m_2,\ell_2)$ using a strong stability preserving Runge-Kutta method.
The low fidelity model uses the same parameterization but a corresponding linear approximation, see Equation (6.7) in \cite{narayan2013stochastic} for details.
In contrast to \cite{narayan2013stochastic}, we use a Euclidean inner product for the examples shown here.

\subsubsection{Setup}
For the high-fidelity model, we use a time step $\Delta t= 10^{-2}$ until $T=15$, resulting in a high-fidelity time-series vector of size $1501$.
The low-fidelity space uses $\Delta t = 0.25$ until $T=15$ resulting in a $61$-dimensional time-series vector.
The uncertain parameter $m_2$ is sampled along the interval $[0.25,0.75]$, and $\ell_2$ is sampled from the range $[0.25,4]$.
We sample the low-fidelity model using $20$ uniform grid samples along each parameter range, resulting in $400$ simulations.

\subsubsection{Analysis}
For this example, we observed a much larger gap between the lowest and highest performers.
This gap indicates that subset selection choice matters more for this problem, probably due to the more strongly non-linear relationship exhibited here than in the Burgers example above.

The error observed in the low-fidelity model, shown in Figure~\ref{fig:pendulum:projerr} on the left, demonstrates that the GOMP-based approach achieves superior performance in almost every subset size, until the number of samples because so large the choice in subset selection does not matter.
Using the low-fidelity subset choices in computing reconstructions for the high-fidelity model, shown on the right side of Figure~\ref{fig:pendulum:projerr}, achieves similar improvements.
In the high-fidelity case, the model errors are much larger making any gains more important.

A random subset selection does quite well on average in comparison to these deterministic approaches.
\edits{However, we emphasize again that in worst-case a random subset can do quite poorly, as shown in the results, which makes it a risky strategy for simulation allocation.}
However we note that rarely would making the allocation of a single high-fidelity simulation based purely on a random selection make sense, due to the risk of any specific subset performing quite poorly.
While in expectation the random case performs well, any specific random subset may not unless the subset size becomes large enough that the subset selection algorithm is no longer important.
However, having arbitrarily large numbers of high-fidelity simulations is typically not possible for real-world situations.

\begin{figure}[tb]
 \centering 
 \includegraphics[width=.49\textwidth]{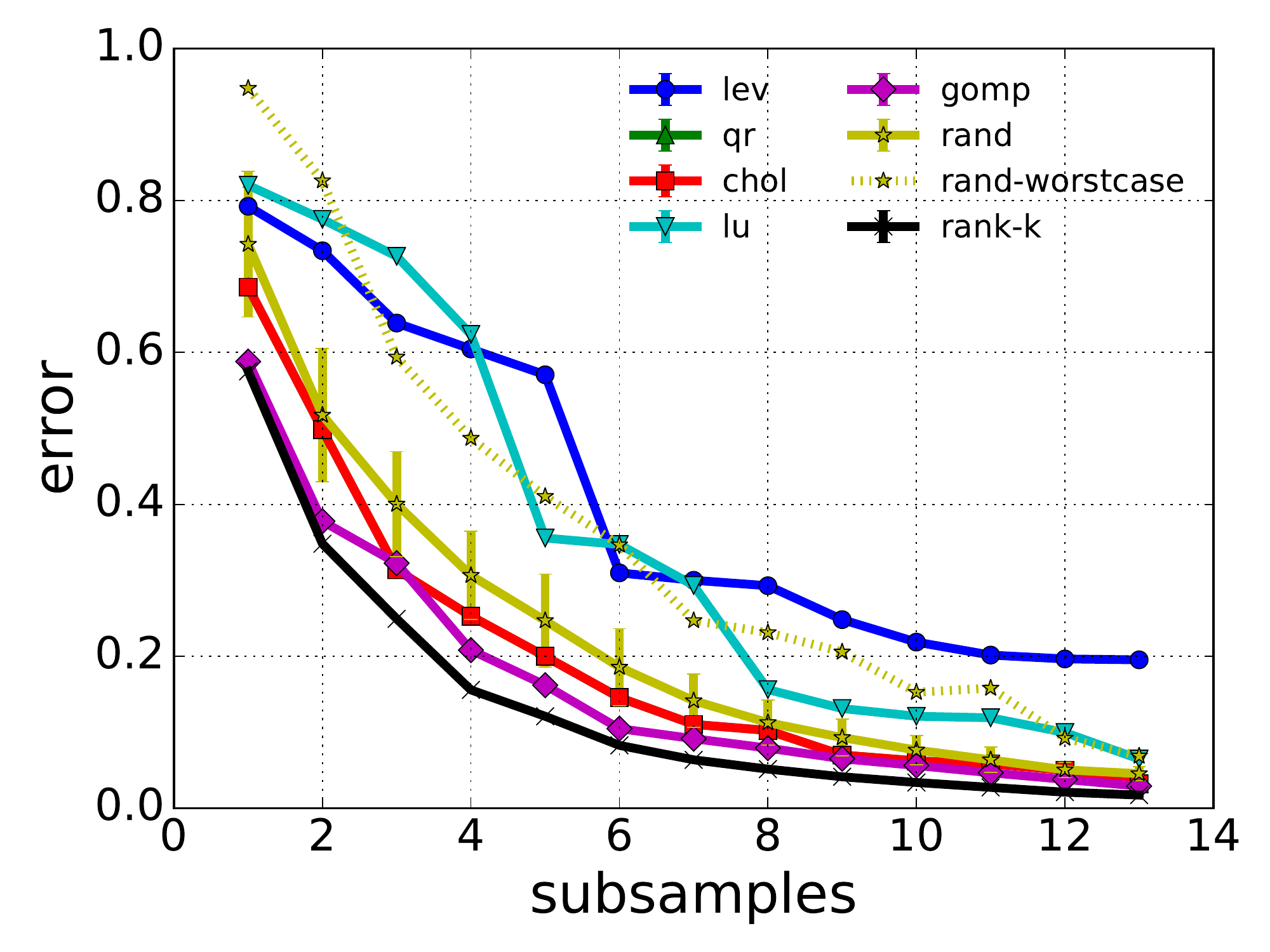} 
 \includegraphics[width=.49\textwidth]{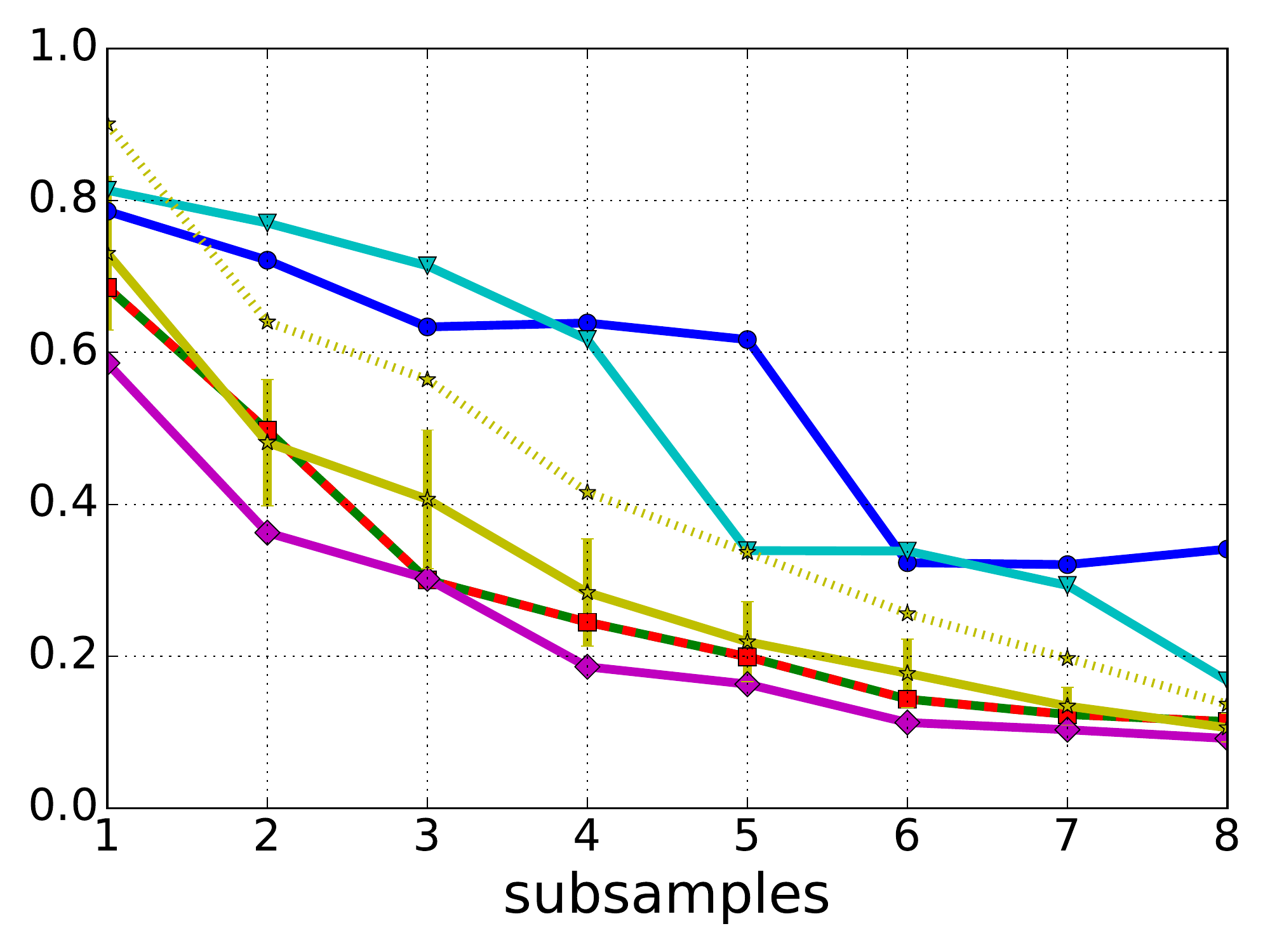} \\
\phantom{A} \hspace{.8in} \textbf{Low-fidelity model} \hfill \textbf{High-fidelity model} \hspace{.8in}
 \caption{%
	\textbf{Double pendulum:} Reconstruction error among subset method for solutions to double pendulum problem.
	\emph{Left:} the reconstruction error of the rest of the low-fidelity simulation dataset when using the indicated subset chosen using various methods.  
	\emph{Right:} the reconstruction error but for the high-fidelity simulations with the subset chosen using the low-fidelity samples.
}
\label{fig:pendulum:projerr}
\end{figure}

\subsection{Compressible flow simulation}
Compressible flow simulations are used to study how fluids act while flowing in and around obstacles.
Understanding these situations in detail is vital in many domains, especially \edits{in engineering scenarios where compressibility effects are important.
We present here results from a 2D compressible flow simulation around a cylindrical object.}
The simulation finds a solution to the compressible Navier-Stokes equation, which for two dimensions can be written as:
\begin{equation}
\frac{\partial q}{\partial t} + \frac{\partial f}{\partial x} + \frac{\partial g}{\partial y} = 0,
\end{equation}
where $q$ is a vector of conserved variables, $f = f(q,\nabla q)$ and $g=g(q,\nabla q)$ are the vectors of the fluxes.
These can be rewritten as:
\begin{equation}
f = f_i - f_v, g = g_i - g_v,
\end{equation}
where $f_i, g_i$ are the inviscid fluxes \edits{as given in \cite{KaSh05}} and $f_v, g_v$ are the corresponding viscous fluxes.
The viscous fluxes take the form,
\begin{equation}
f_v = \left( \begin{matrix} 0 \\ \tau_{xx} \\ \tau_{yx} \\ u \tau_{xx} + v \tau_{yx} + k T_x \end{matrix} \right),\
g_v = \left( \begin{matrix} 0 \\ \tau_{xy} \\ \tau_{yy} \\ u \tau_{xy} + v \tau_{yy} + k T_y \end{matrix} \right),
\end{equation}
where $\tau$ is the stress tensor,
\begin{equation}
\tau_{xx} = 2 \mu (u_x - \frac{u_x + u_y}{3}), \tau_{yy} = 2 \mu (v_y - \frac{u_x + u_y}{3}), \tau_{xy} = \tau_{yx} = \mu (v_x + u_y),
\end{equation}
with $\mu$ the dynamic viscosity and $k$ the thermal conductivity.
The corresponding solution is well known and we refer the reader to classic texts for appropriate background, such as \cite{munson1990fundamentals}.

We are specifically interested in two uncertain parameters, the Reynolds number, $Re$, which relates the velocity of the fluid to the viscosity of the fluid, and the Mach number, $Ma$, which relates the speed of the fluid to the local speed of sound.
We ran the simulations using the compressible flow module of the Nektar++ suite \cite{cantwell2015nektar++}.
In the simulation, we use the following relationships between the Reynolds number and dynamic viscosity, and between the Mach number and the \emph{farfield} velocity:
\begin{equation}
\mu = \frac{\rho_{\infty} u_{\infty}}{Re}, \,\,\,\,\,\,u_{\infty} = Ma \sqrt{\frac{\gamma p_{\infty}}{\rho_{\infty}}},
\end{equation}
where $p_{\infty}, \rho_{\infty}, u_{\infty}$ denote the \emph{farfield} pressure, density, and x-component of the velocity, respectively, and $\gamma$ is the ratio of specific heats.
An example of a single time point in one of the simulations is shown in the left side of Figure~\ref{fig:cylinder:flow}.

\begin{figure}[tb]
\centering
\includegraphics[width=.45\textwidth]{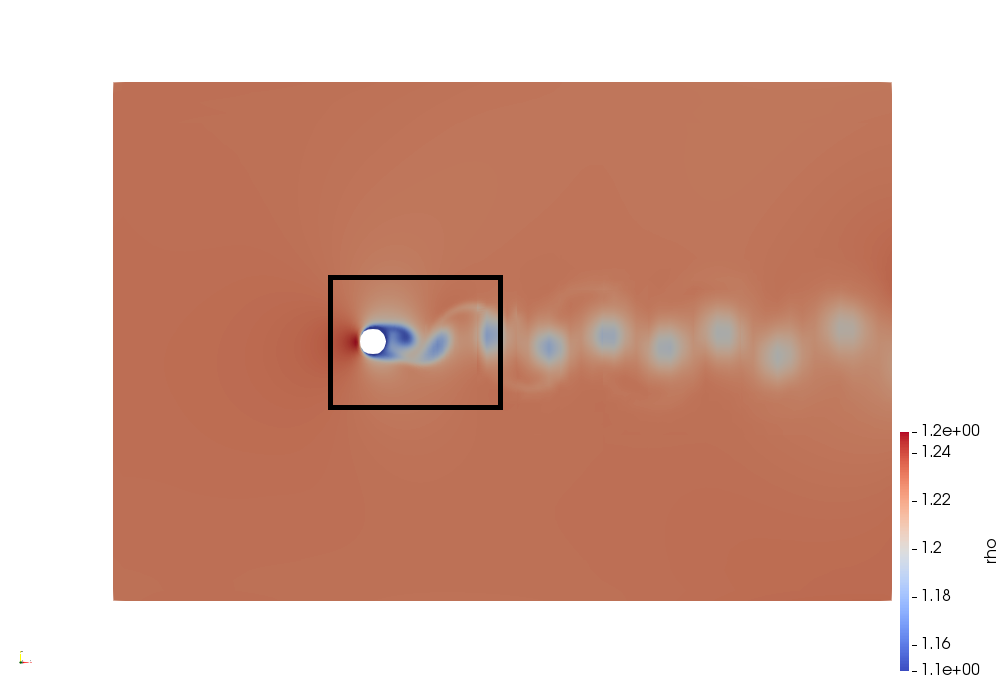}
\hspace*{.2in}
\raisebox{0.15\height}{\includegraphics[width=.30\textwidth]{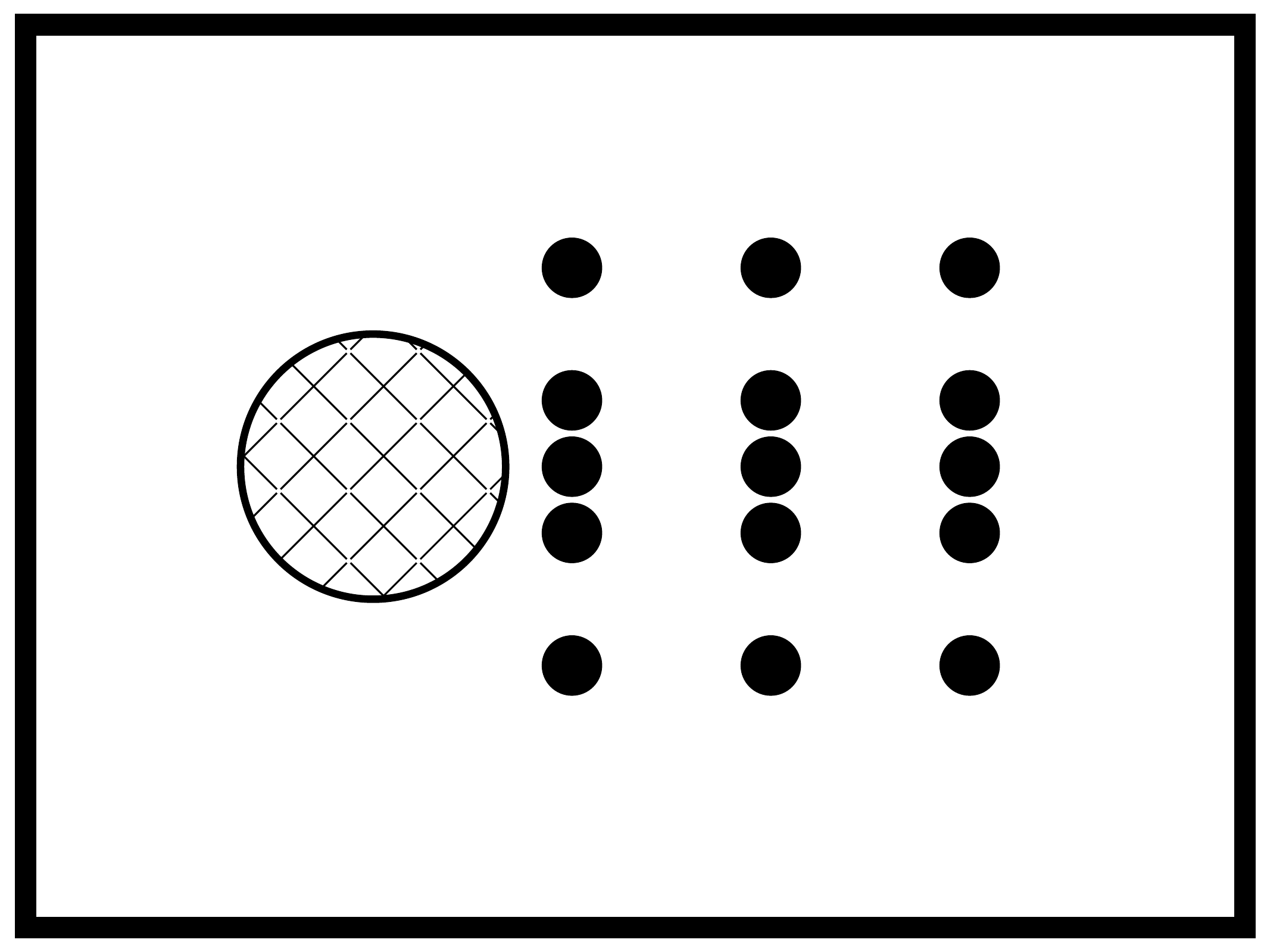}}
\caption{%
\textbf{Compressible flow simulation}.
\textit{Left:} An example of one of the compressible flow simulations used in this experiment, showing the entire field of mass-density ($\rho$) values.
The simulation models a fluid flowing past a cylinder, we are specifically interested in the mass-density and energy values that emerge in the immediate wake of the cylinder.
\textit{Right:} A zoomed-in diagram showing the location of the 15 sample points taken from each simulation for the analysis.  
These points were specifically positioned in the immediate wake of the cylinder.
The location of this diagram is shown by the rectangle in the image on the left.
}
\label{fig:cylinder:flow}
\end{figure}

\subsubsection{Setup}
We evaluated the Reynolds number, $Re$, and Mach number, $Ma$, parameters at regular intervals with $5$ samples between $200$ and $450$, and $6$ samples between $0.2$ and $0.45$ respectively.
For a single parameter pair, we ran the simulation for $1 \times 10^7$ steps to achieve $1$ unit of non-dimensionalized simulation time, and for each time-step we recorded $4$ values at $15$ spatial locations in the immediate vicinity of the cylinder.
The spatial locations are shown visually in Figure~\ref{fig:cylinder:flow}.
We are primarily interested in the fluid flow simulation after it transitions into a steady shedding regime, so we ignore any of the time steps in the startup regime, which for this problem we observed to be at 30\% of the steps, or $4 \times 10^4$ steps.
This results in a feature vector with $6 \times 10^4$ time-steps and $15 \times 4 = 60$ values for each step.

Once the simulation transitions from the transient regime, each of these simulation sample points are periodic signals whose form depends nonlinearly on the parameters chosen.
To align the signals we detect the period length, find the cyclic peak, and translate all signals so that the first peak aligns.
We then crop all signals to a single period size.
The cropping was done to ensure a nice decay of singular values in the ensemble.
A sample of a single positional point and variable but for various Reynolds and Mach numbers is shown in Figure~\ref{fig:cylinder:signal} for reference.
The corresponding Hilbert space is induced by the inner produce of these period length signals.

We ran both the low- and high-fidelity simulations over the specified parameter values in order to compute the proxy approximation error using the various subset selection methods.

\begin{figure}[tb]
\centering
\includegraphics[width=.9\textwidth]{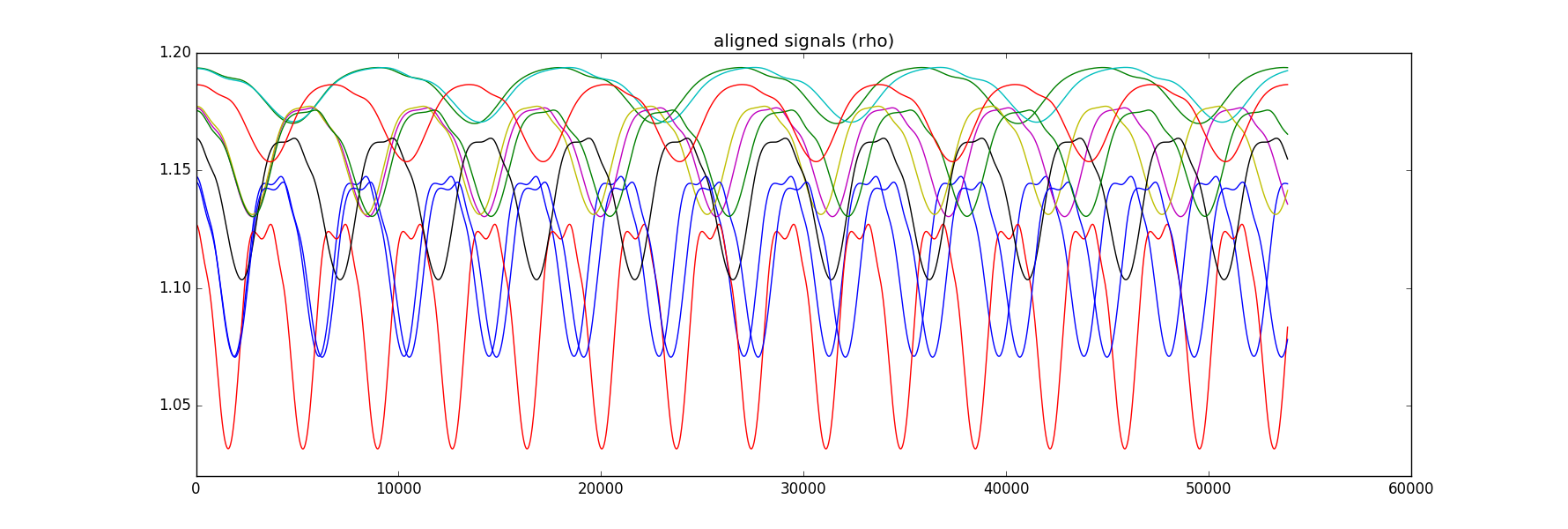}\\
\includegraphics[width=.9\textwidth]{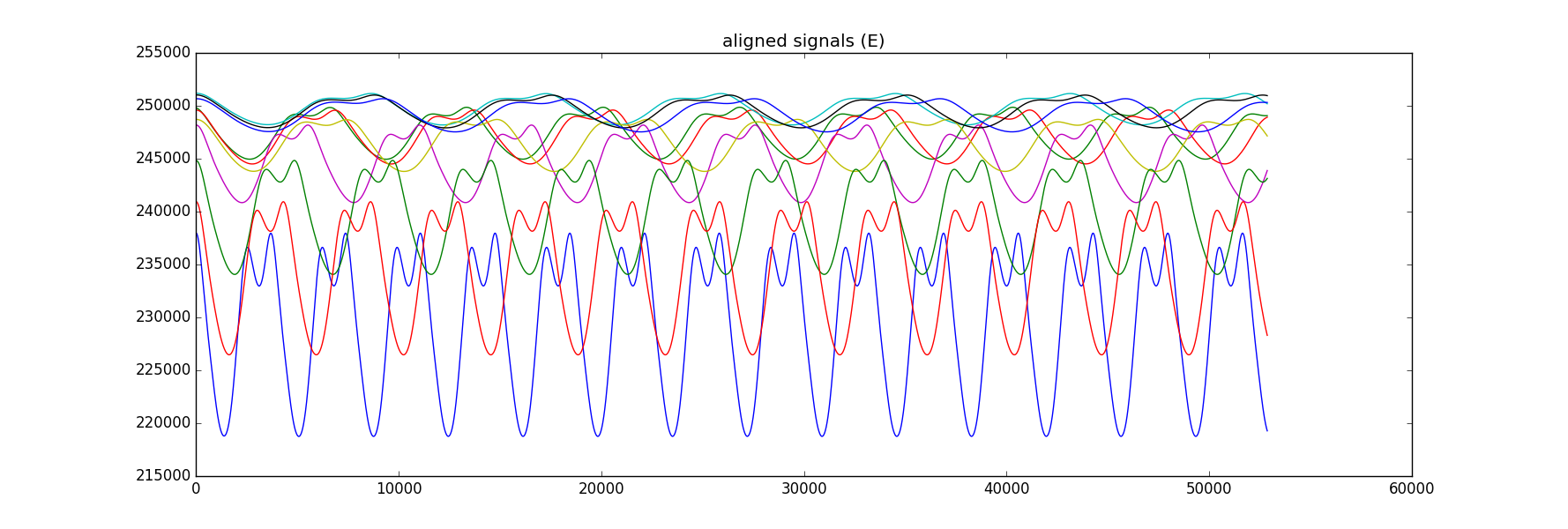}\\
\caption{%
\textbf{Compressible flow simulation point samples over time}.
\textit{Top:} Samples of the mass-density field at a single positional point $(1.5,0.0)$, which is located in the immediate wake of the cylinder in Figure~\ref{fig:cylinder:flow}.  The differing signal shapes are due to changes in the Reynolds and Mach numbers of the simulation, the figure shows 10 randomly samples Reynolds-Mach pairs.
\textit{Bottom:} Same as above but from the Energy field.
}
\label{fig:cylinder:signal}
\end{figure}

\subsubsection{Analysis}
The results of the comparison in terms of how each of the subset selection methods perform is shown in Figure~\ref{fig:cylinder:projerr}.
Here we observe slightly differing results in the high-fidelity vs low-fidelity errors.
In the low-fidelity case the methods perform similar to previous experiments.
The GOMP, Cholesky, QR and leverage methods all perform the same after a certain number of samples, and all perform better than random sampling.
For the high-fidelity results we observed that Cholesky and QR perform similarly to random, while GOMP performs the best.
One of the primary differences between GOMP and QR is that GOMP performs selection based on residual correlation in the Hilbert space, while QR uses residual magnitude.
These results indicate that the correlation becomes more indicative than magnitude of high-fidelity sample importance for this type of simulation data; this is another reason to consider GOMP when choosing the subset selection technique.

\begin{figure}[tb]
 \centering 
 \includegraphics[width=.49\textwidth]{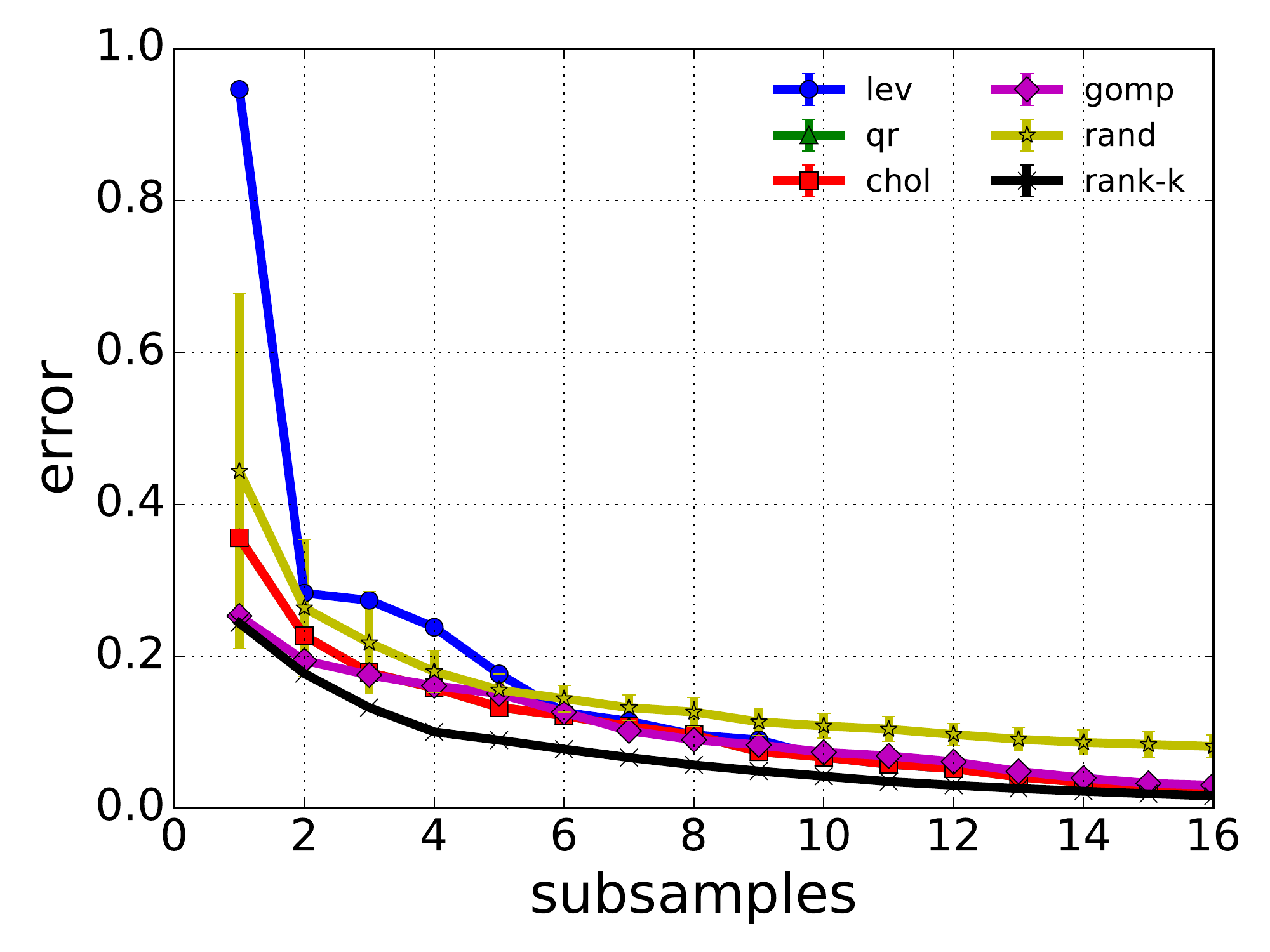} 
 \includegraphics[width=.49\textwidth]{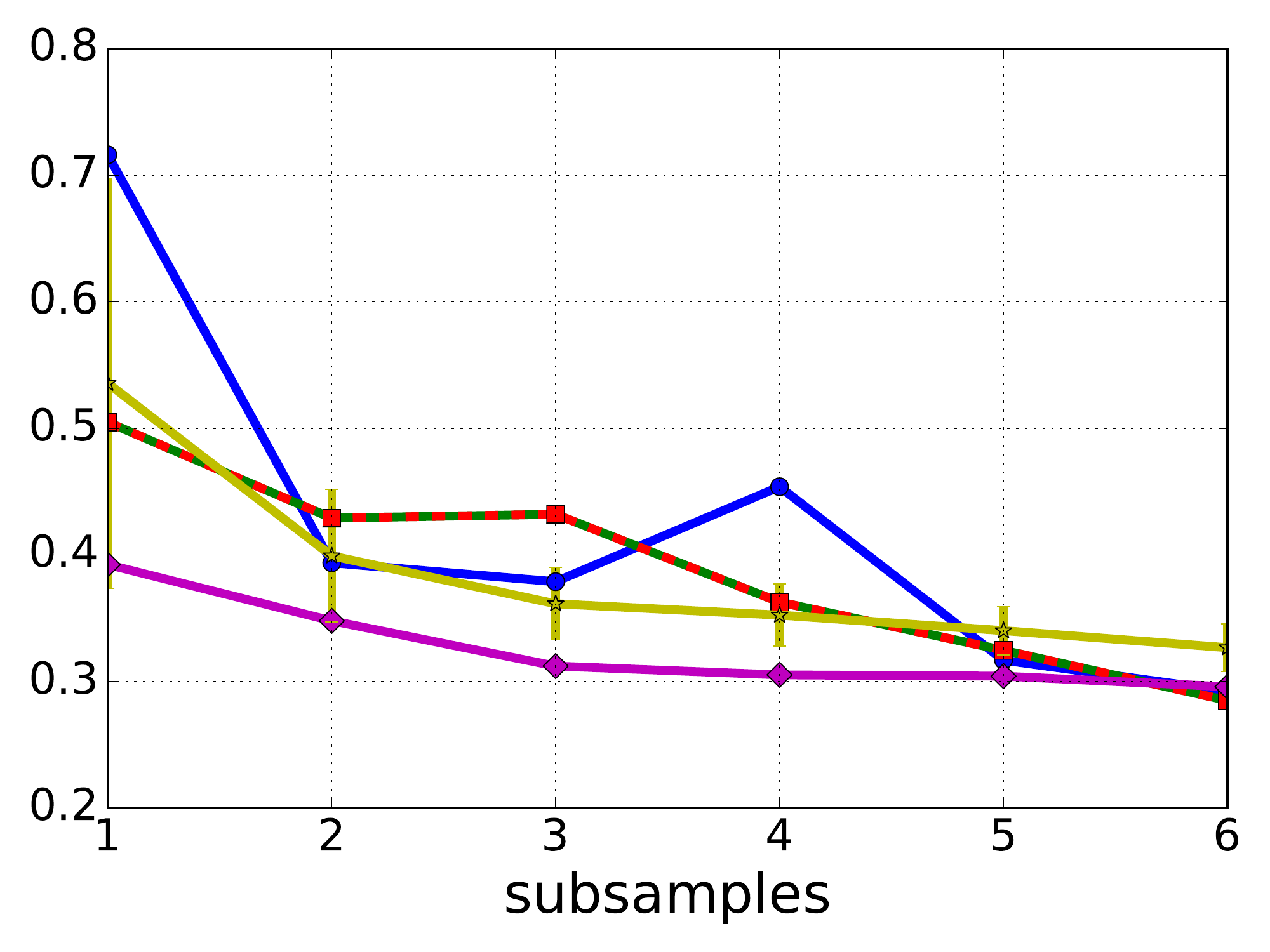} \\
\phantom{A} \hspace{.8in} \textbf{Low-fidelity model} \hfill \textbf{High-fidelity model} \hspace{.8in}
 \caption{%
	\textbf{Projection error among subset method for solutions to compressible flow problem.}
	\emph{Left:} the reconstruction error of the rest of the low-fidelity simulation dataset when using the indicated subset chosen using various methods.  
	\emph{Right:} the reconstruction error for the high-fidelity simulations with the subset chosen using the low-fidelity samples.
}
\label{fig:cylinder:projerr}
\end{figure}

\subsection{Structure topology optimization}
A structure topology optimization (STO) problem involves solving a non-convex optimization problem for the optimal material placement to satisfy a physical property such as stiffness or maximum stress.
STO has emerged as a powerful tool in designing various high performance structures, from medical implants and prosthetics to jet engine components \cite{sutradhar2010topological,reist2010topology,zhu2016topology}. 
The optimization shown here is a canonical STO problem in which the optimization finds the best material layout in the design space in order to maximize the stiffness (or equivalently minimizing the deflection or compliance) of the structure subject to a material volume constraint. The response of the structure for a set of loading and boundary conditions is typical computed/simulated via finite element method (FEM). The topology optimization yields a solution in the form of binary maps that indicate material placement.
\eedits{
  The low- and high-fidelity models are binary map outputs, see Figure \ref{fig:topo:diff} left. Our multifidelity procedure does not operate on the binary maps themselves, but instead on a signed distance transform of the binary maps. To perform linear approximations, we compute signed distance transform (SDT) fields from binary maps and use these fields in the linear reconstruction. The SDT fields can be thresholded to recover binary maps, see Figure \ref{fig:topo:diff} right for examples. All errors reported in this section are errors in the SDT fields. Figure \ref{fig:topo:diff} shows that thresholded SDT reconstructions can provide some macroscopic information about structures, but can lose a substantial amount of small-scale structure.
}
Additional details of the specific optimization and example implementations can be found in \cite{andreassen2011efficient}.

The problem considered here is parameterized by three variables, the vertical position of the loading, $p$, the loading angle, $\theta$, and the filter size, $\rho$, which controls the scale of contiguous blocks in the material placement.
This is summarized in Figure~\ref{fig:topo:diagram}.

\begin{figure}[tb]
\centering
\includegraphics[width=.49\textwidth]{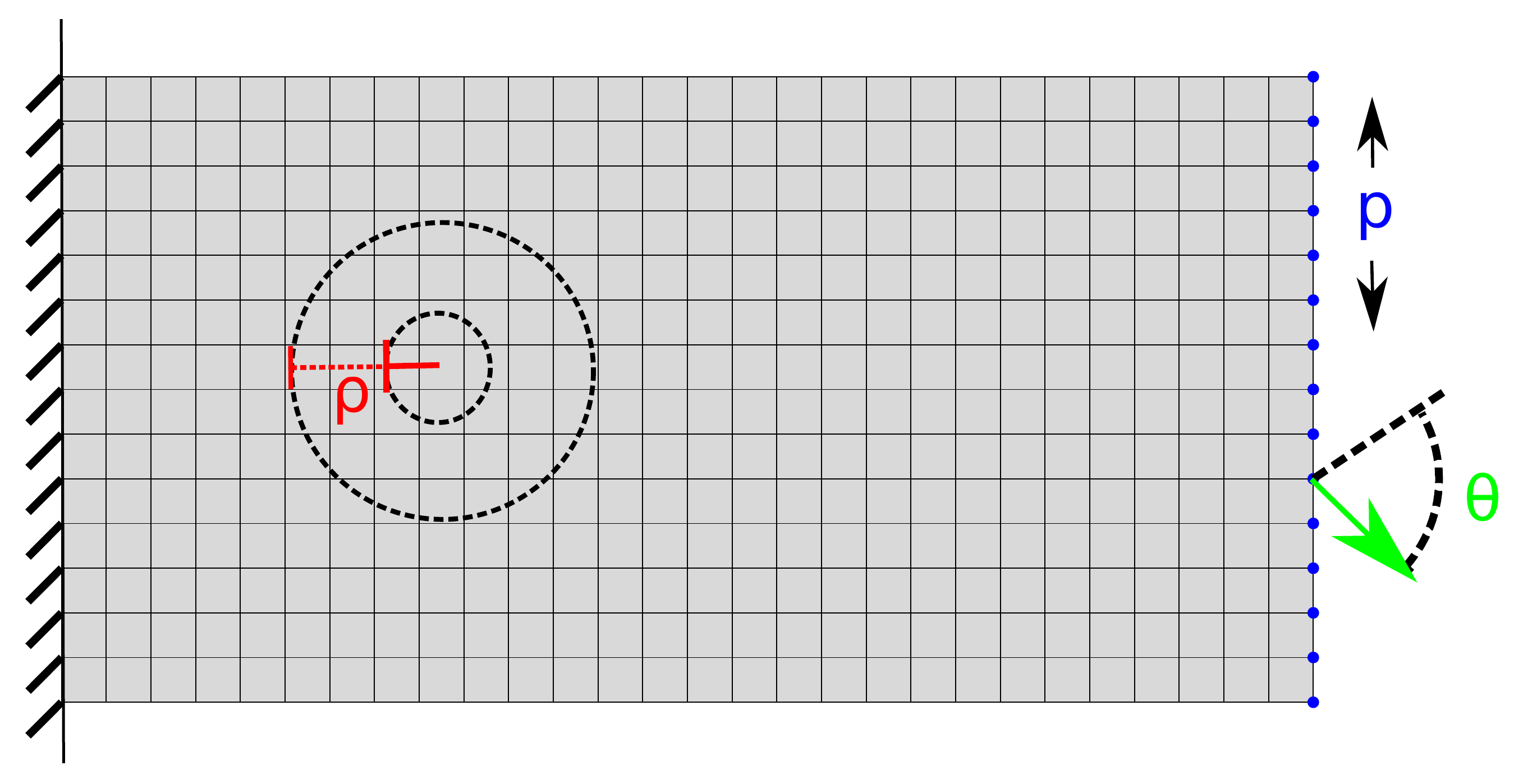}
\caption{%
\textbf{Structure topology optimization} diagram.
The loading on the right edge is parameterized by a vertical position, $p$, and a loading angle, $\theta$, and the filter size, $\rho$, which controls the length scale in the topology optimization.
}
\label{fig:topo:diagram}
\end{figure}

\subsubsection{Setup}
The topology optimization takes place on a 2D rectangular grid of size $n_{\text{row}} \times n_{\text{col}}$.
For this experiment, the position variable, $p$, was restricted to $p \in [-0.5n_{\text{row}},0.5n_{\text{row}}]$ corresponding to the extent of the right boundary depicted in Figure~\ref{fig:topo:diagram}.
The angle parameter was restricted to $\theta \in [0,\pi]$ corresponding to the full range between pointing straight down to straight up.
The filter parameter was restricted to $\rho \in [1.1z,2.5z]$ corresponding to the scale of features allowed in the final simulation, where $z$ is the appropriate ratio coefficient depending on the resolution $s_x,s_y$.
We sampled the three parameter spaces in a uniformly random way until $1000$ sample triples were found.
Those triples were then used to run $1000$ separate topology simulations.
Low-fidelity corresponds to the problem solved with $n_{\text{row}} = 40, n_{\text{col}} = 80, z=1$, and high-fidelity corresponds to the problem solved on a much larger region, $n_{\text{row}} = 80, n_{\text{col}} = 160, z=2$.

\subsubsection{Analysis}
Our analysis takes into account both the low-fidelity and high-fidelity spaces.
Specifically we analyze how well the subset selected by each method works to reconstruct the remaining low-fidelity samples.
We also consider how the method reconstructs the high-fidelity samples using the same subset.

We observed that the GOMP approach performed the best for any subset size examined.
The QR and Cholesky methods also performed quite well in the low-fidelity reconstruction error.
The average of random method performed relatively well initially, but then all other method perform better at larger subset sizes.

The high-fidelity case was more different for this dataset than in the previous simulations examined.
Specifically we note the GOMP maintains the position as the best performer, however it also gains a considerable edge in the larger subset sizes.
Surprisingly, Cholesky and QR perform well at first, but then perform worse as the subset size grows.
The average of random subsets does better than the LU and leverage approaches.
The performance comparison among the different methods is summarized in Figure~\ref{fig:topo:projerr}.

The poorer performance of QR for larger subset sizes, we hypothesize, is due to the larger set of differences between the low- and high-fidelity solutions in this dataset.
Because of the non-convex nature of the topology optimization problem, we observe some larger differences in the various solutions, an example is shown in Figure~\ref{fig:topo:diff}.
We hypothesize that because GOMP bases the decision of using a subset element largely on overall correlation with the remaining data point, rather than simple residual error magnitude, it is able to capture the more important trends of the samples which carry over to the high-fidelity case.

\begin{figure}[tb]
 \centering 
 \includegraphics[width=.3\textwidth]{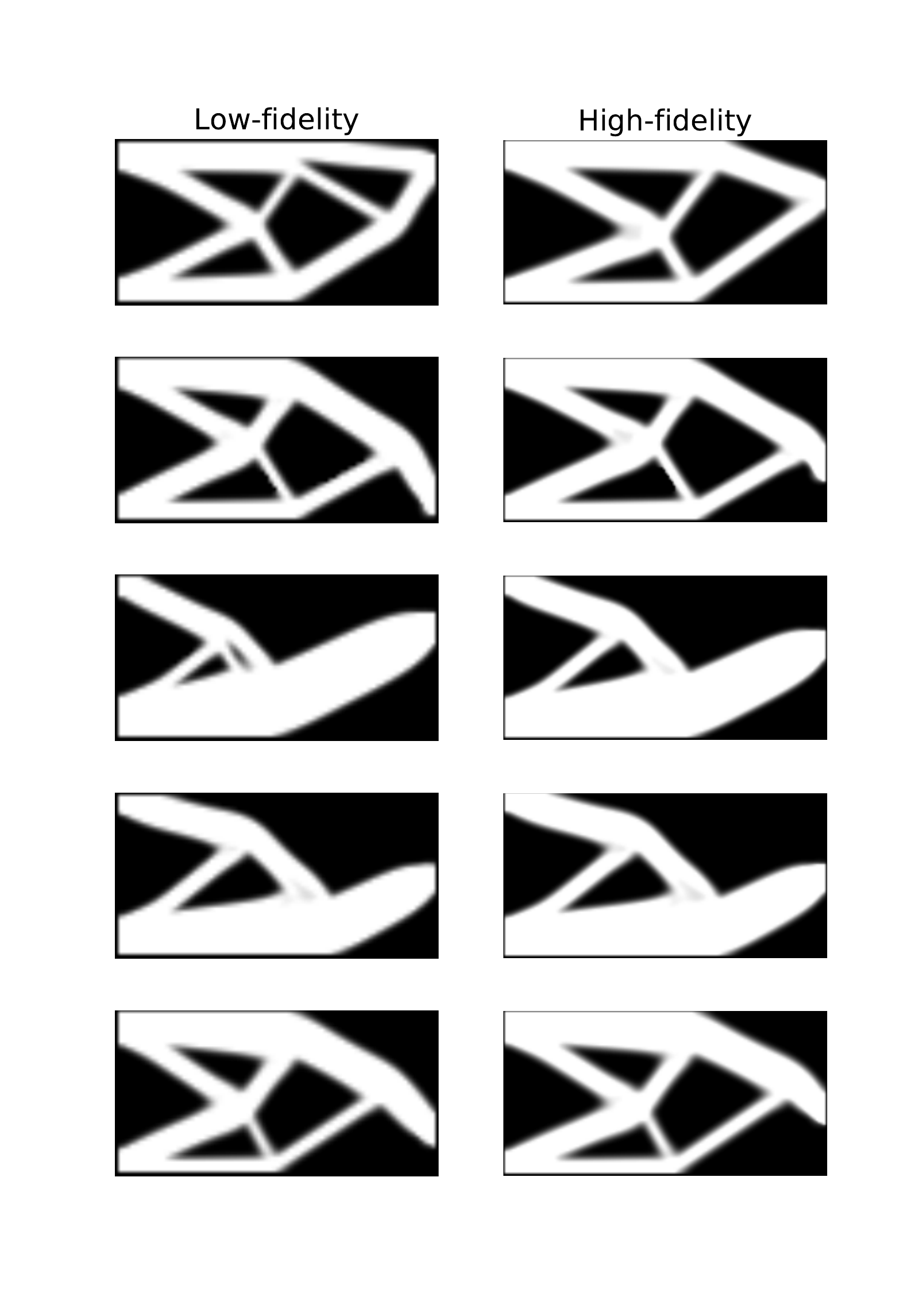}
 \includegraphics[width=.69\textwidth]{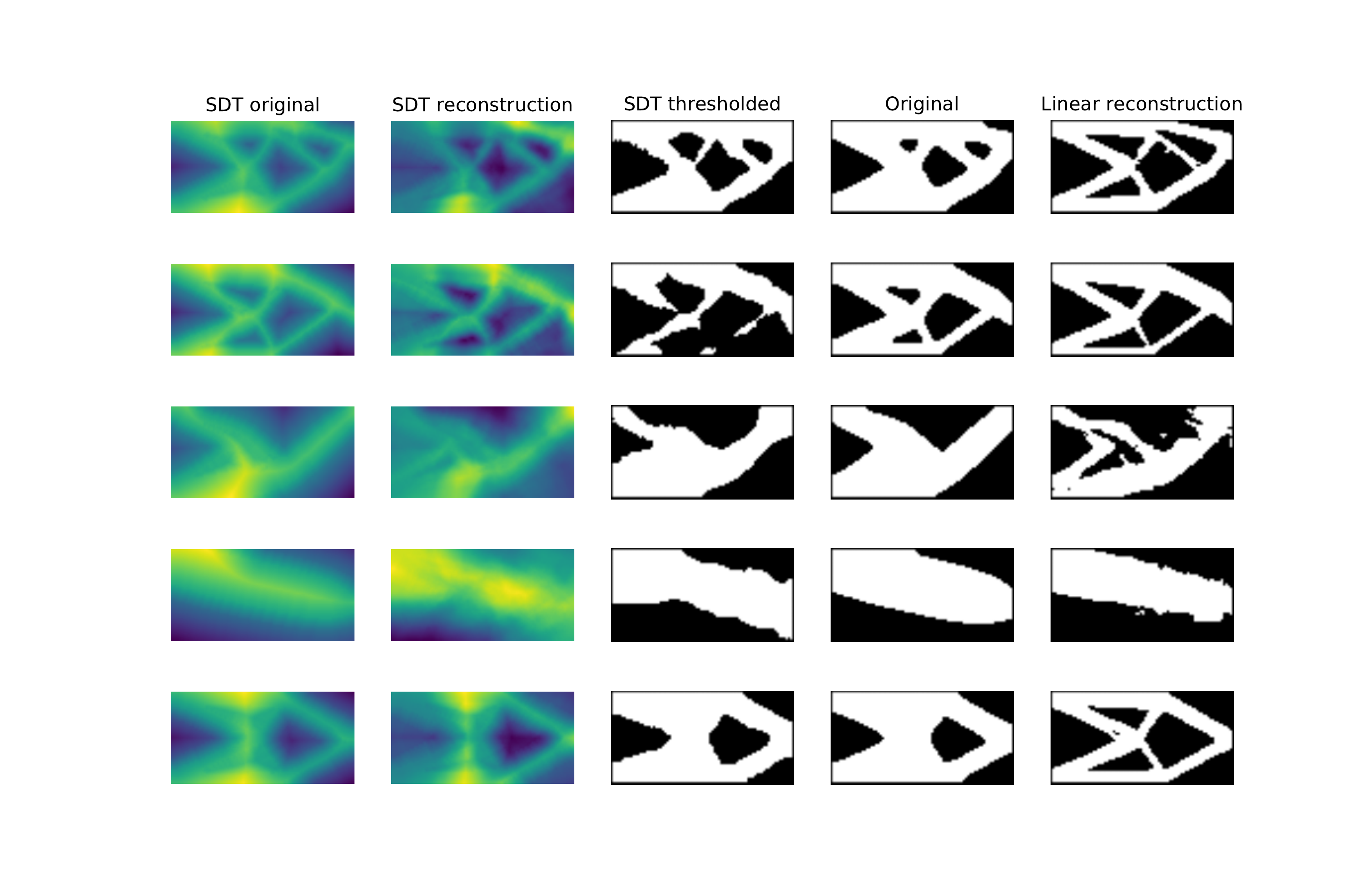}
 \caption{%
   \textbf{Structure topology optimization} \eedits{
        \textit{Left 2 columns:} Binary maps output from both low- and high-fidelity solutions.
   Each row shows the low- and high-fidelity solutions using the same parameters, note the variety of changes in the solution just by moving from a lower resolution domain to a higher resolution domain. \newline
     \textit{Right 5 columns:} Illustration of signed distance transform (SDT) fields resulting from binary maps output from the optimization process. The first column shows the SDT fields from the original binary maps (fourth column). The second column shows the SDT reconstruction, which is subsequently thresholded to obtain the binary structure (fifth column). For reference, the third column shows thresholding of the first-column SDT field.
 }
}
\label{fig:topo:diff}
\end{figure}

\begin{figure}[tb]
 \centering 
 \includegraphics[width=.49\textwidth]{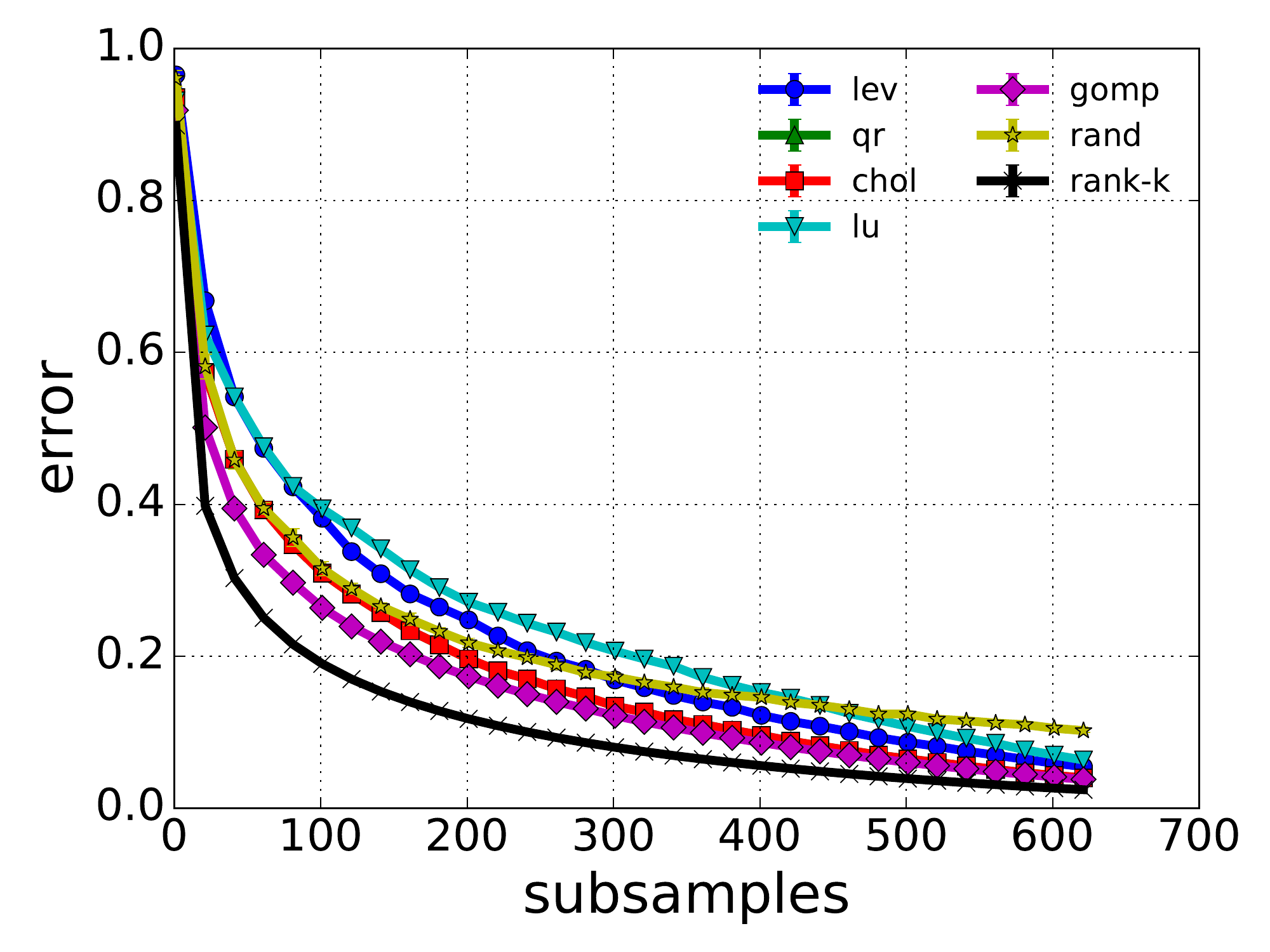} 
 \includegraphics[width=.49\textwidth]{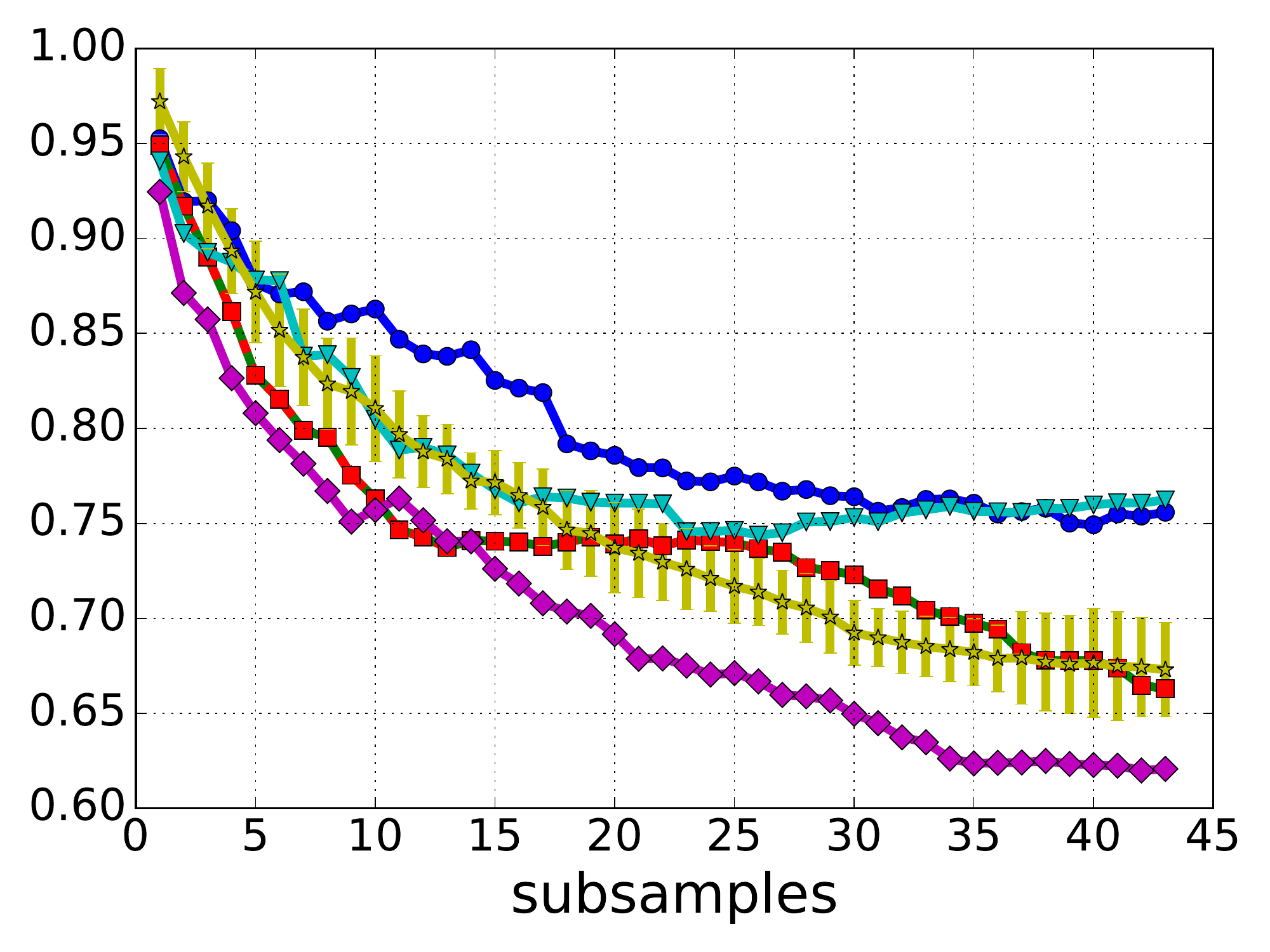} \\
\phantom{A} \hspace{.8in} \textbf{Low-fidelity model} \hfill \textbf{High-fidelity model} \hspace{.8in}
 \caption{%
	\textbf{Structure topology optimization:} Reconstruction error among subset methods.
	\emph{Left:} the reconstruction error of the rest of the low-fidelity simulation dataset when using the indicated subset chosen using various methods.  
	\emph{Right:} the reconstruction error for the high-fidelity simulations with the subset chosen using the low-fidelity samples.
}
\label{fig:topo:projerr}
\end{figure}

Our previous analysis has focused on the subset size.
A larger subset size for the same error ultimately translates into a higher runtime cost to obtain the same approximation error.
To make this relationship more clear, we also compared the methods directly using simulation runtime and reconstruction error.

A popular alternative approach to generating proxy functions for applications in uncertainty analysis of simulations is to use a Gaussian process (GP) regression to estimate unseen simulations results (see \textit{e.g.} \cite{le2013multi,perdikaris2015multi,perdikaris2017nonlinear}).
However, a GP approach relies on a large number of high-fidelity simulations from which to derive the proxy function.
This is directly at odds with computationally constrained high-fidelity problems.
We illustrate this point by including a GP proxy using the same, \emph{small} number of samples used in the multifidelity approach.

We also consider an alternative neural-network-based proxy-method proposed specifically for topology optimization in \cite{ulu2016data}.
In \cite{ulu2016data}, the authors proposed training a standard feed-forward neural network, also known as a multilayered perceptron (MLP), as a proxy for the exact solution given the problem constraints.
They first reduce the dimension of the dataset to an $80$-dimensional space using PCA, and then train the MLP to regress to the PCA weights given the problem constraints.
We implemented their approach and get comparable results for the same number of solution designs used in their training ($400$).
However, this approach also assumes the availability of sufficient data to train the model, which is not always the case in the high-fidelity models in the wild.
To compare we used both an MLP based approach with no PCA reduction, and an MLP training on the PCA weights where the space was chosen to capture 90\% of the singular value energy.
The results are summarized in Figure~\ref{fig:topo:cost}.
The left side of Figure~\ref{fig:topo:cost} compares all the methods on a log-scale.  
Because of the lack of training data, the MLP model performs quite poorly.
The GP model performs much better initially than the MLP model on the fewer number of samples, but after reaching between $10--15$ samples the two become quite similar in performance.
Both the GP and MLP models accrue considerable error because the limited samples do not sufficiently represent the high fidelity space. 
\edits{The right side of Figure~\ref{fig:topo:cost} shows only the multifidelity methods on a linear-scale.  
Note that it is difficult to directly compare results here with those in Figure~\ref{fig:topo:projerr}, because this plot shows an average cost and runtime over $50$ smaller datasets of size $500$ each, which results in much smaller error overall. However the relative performance of the methods is generally the same.}

In contrast the Gramian-weighted non-parametric proxy method introduced in \cite{narayan2013stochastic} performs well using only a small number of high-fidelity samples because it makes direct use of the low-fidelity structure through the low-fidelity Gramian.
In considering the choice in how the subset selection of the non-parametric proxy method is done, we found similar advantages in using a GOMP-based subset to the previous examples considered.
Specifically we note that given the same amount of high-fidelity simulation time we can improve the error considerably over previous methods.
\edits{Our comparison in Figure \ref{fig:topo:cost} between our multifidelity methods (lev, qr, chol, lu, gomp) and surrogate methods (MLP, MLP-pca, GP) is not quite a comparison on an even playing field: The cost of training the multifidelity and surrogate methods are comparable, but evaluating the multifidelity approximations requires a low-fidelity evaluation, which can be more expensive than evaluation of a neural network or GP surrogate. However, this example illustrates that in our situation neural networks and GP's cannot be trained on only high fidelity data with reasaonble cost. In scenarios when the high-fidelity model is expensive, it is thus more difficult to accurately train surrogate models, and it is in this regime when the multifidelity approximation can be useful.}

\begin{figure}[tb]
 \centering 
 \includegraphics[width=.45\textwidth]{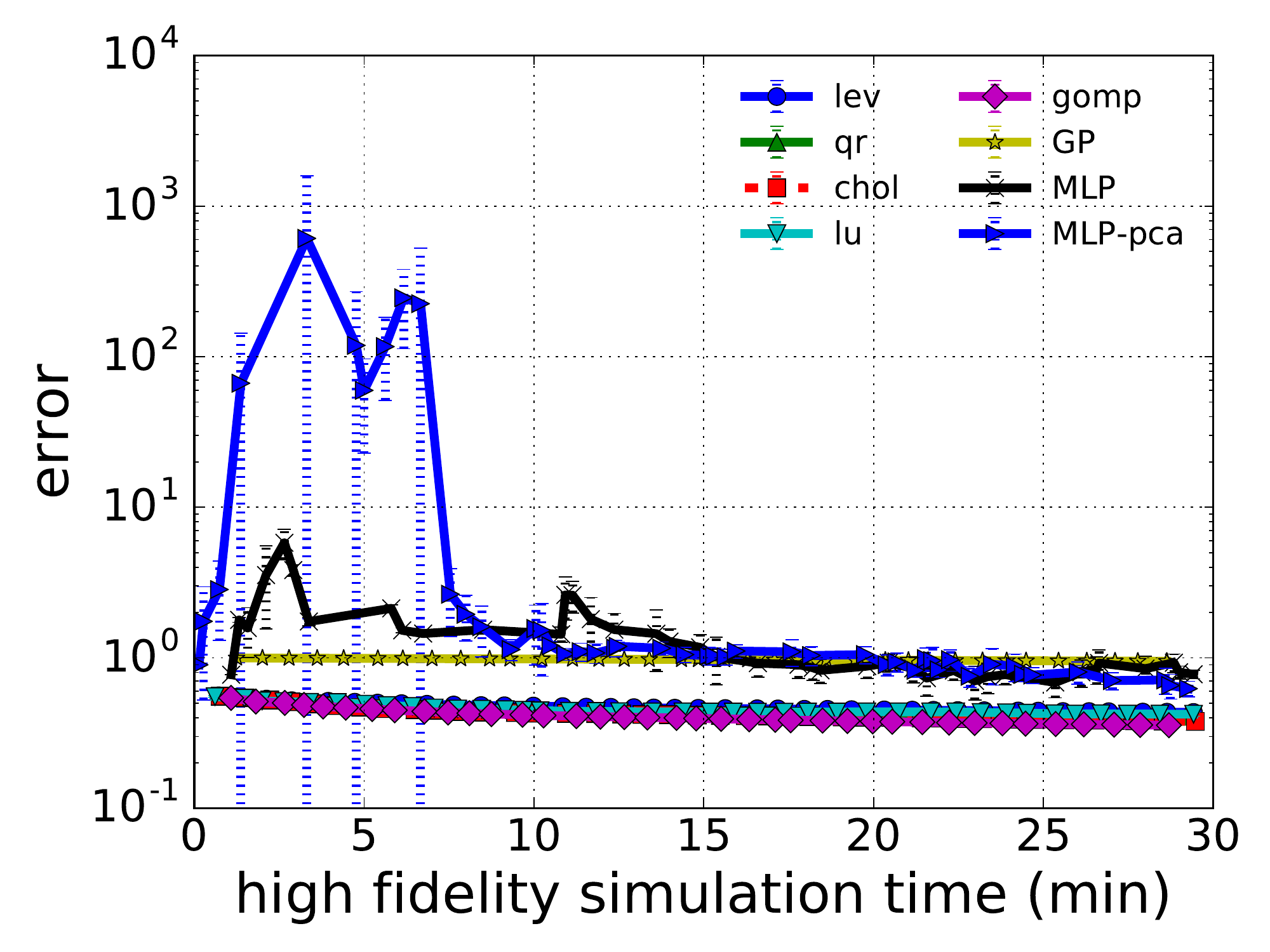} 
 \includegraphics[width=.45\textwidth]{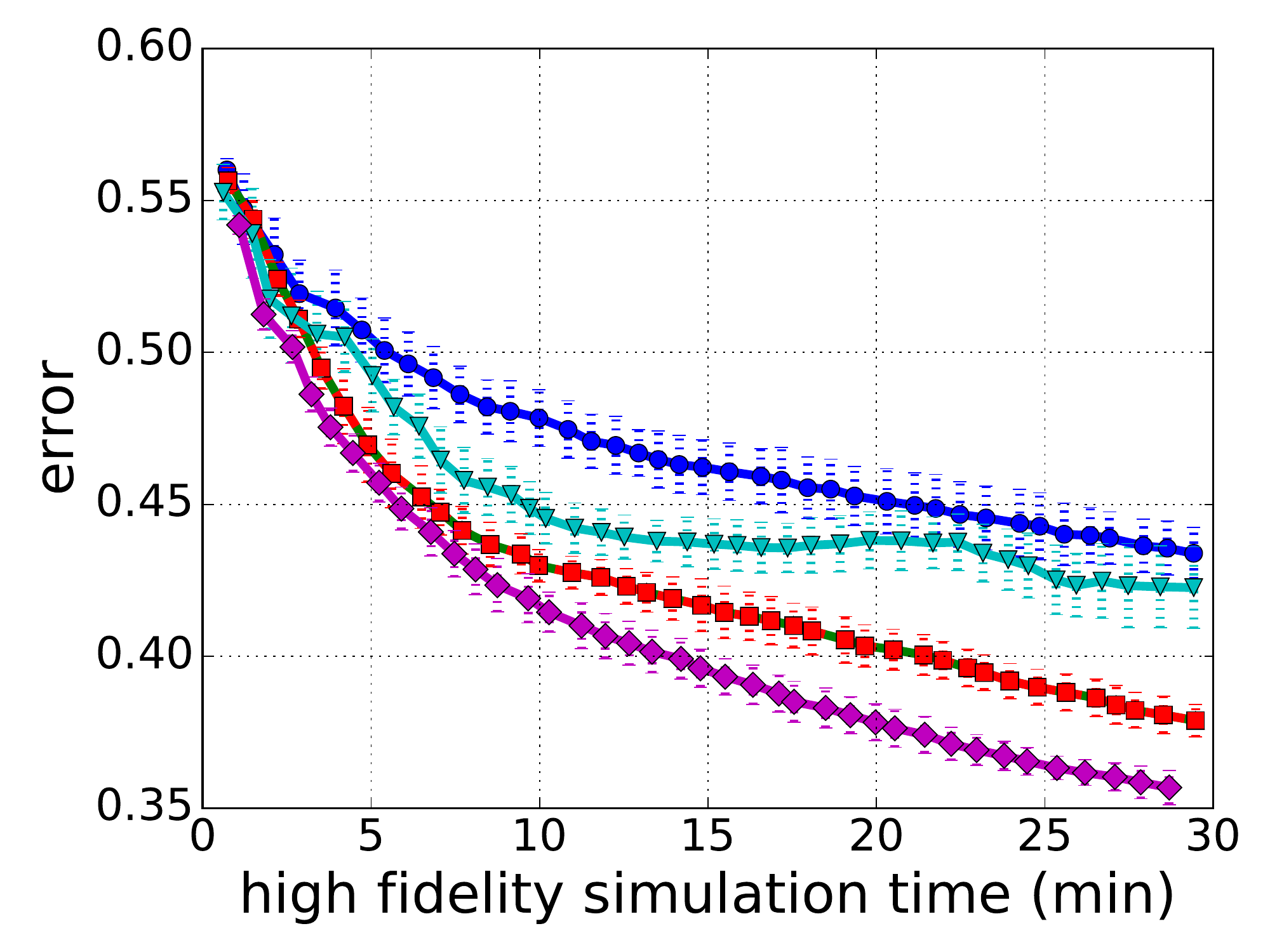} 
 \caption{%
	\textbf{Structure topology optimization} 
\edits{
    Average reconstruction error (vertical axis) versus high-fidelity simulation runtime (horizontal axis) among subset methods over an ensemble of $50$ random datasets containing $500$ samples each; the standard deviation of the ensemble error is shown in dashed error bars around each point. The ``error" is the relative $\ell^2$ vector error averaged over all samples.
   From left to right each point indicates the number of high-fidelity simulations used to obtain the approximation and the vertical and horizontal location is the average error and average time to simulation, respectively, over the $50$ dataset realizations for that number of simulations.
    Left: A log-scale comparison including a Gaussian process (GP), multilayered perceptron (MLP), and an MLP training on PCA projections (MLP-pca) similar to \cite{ulu2016data}.
    Right: A linear-scale view of the same results but only showing the multifidelity techniques.
}
}
\label{fig:topo:cost}
\end{figure}

There are some significant differences between the observed reconstruction error performance among the different subset methods.
For the topology optimization problem, these differences can be seen easily by viewing the subsets of solutions chosen and considering the differences.
Figure~\ref{fig:topo:samples} shows the first $15$ subsets chosen by each of the methods in order of selection.
Note that QR and Cholesky chose the same subsets for this dataset, and for brevity we only show QR.
All of the methods select a similar solution initially, one of the solid beam structure solutions.
This is consistent with our understanding of each of the methods, as the leverage method is selecting the sample with the most statistical leverage, QR, Cholesky, and LU are selecting the sample with the large residual, and GOMP is selecting the sample with the largest correlation with all the samples. 
Note how QR selects a sample with some slant, this probably results in a larger magnitude image, while GOMP selects a beam with a straight orientation, this is probably higher correlated with more samples.
The methods start to differentiate after a few samples: statistical leverage continues to select samples with similar large beam structures because those solutions also have a large statistical leverage score (leverage scores do not take into account prior samples like the other methods).
The order of the QR and Cholesky selections are interesting because, as the first sample was slanted one way, the residual error indicated a slant the other way was the next best according to the magnitude criteria.
After the slanted beams, other solutions with large beams in them as well as small lattices were chosen, probably because these also resulted in large residual magnitudes.
The GOMP approach, in contrast, selected two large beam solutions, and then next selected a solution made entirely of lattices, probably because that distribution of material more highly correlated with the remaining residuals.
We hypothesize that because GOMP uses residual correlation instead residual magnitude, it is able to do better in the high-fidelity space even if there is some differences between low- and high-fidelity solutions for the topology optimization problem.

\begin{figure}[tb]
 \centering 
 \includegraphics[width=.9\textwidth]{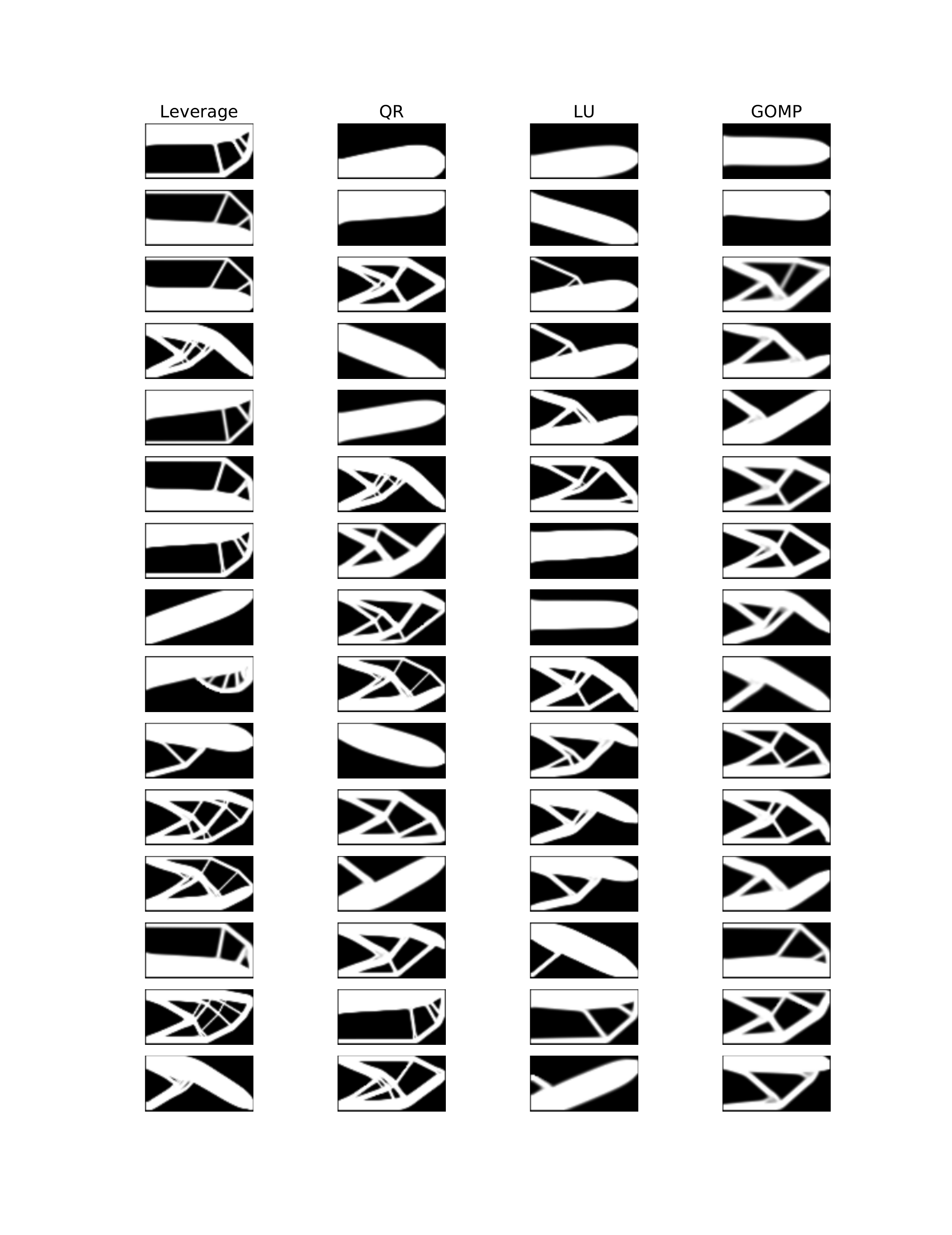} 
 \caption{%
	\textbf{Structure topology optimization samples:} 
    The samples chosen (in order) by each algorithm. Because Cholesky and QR select the same samples, we only show the results for QR in the interest of space.
    Note that all methods select samples with a large primary beam first, and then become quite different afterwards.
    The differences are caused by how the methods make the next selection, for example GOMP uses max correlation with residual while QR/Cholesky use max residual magnitude.
}
\label{fig:topo:samples}
\end{figure}

\section{Conclusion}
We have \eedits{investigated} a novel approach to \eedits{determine a design of experiments} for high-fidelity simulations in a multifidelity \eedits{framework}. \eedits{Low-rank multifidelity techniques require a step that identifies an appropriate allocation of scarce computational resources in a parametric high-fidelity model. 
Existing methods accomplish this allocation by using a particular greedy technique to select a small subset from a large candidate. For the purposes of performing this multifidelity approximation, we explore and compare many existing alternative subset selection techniques that are popular in statistics and machine learning. We find that a GOMP algorithm can consistently yield superior results on challenging datasets.}

\eedits{We provide first-principles motivation of the problem and the choice of GOMP, and present analysis suggesting that GOMP can be effective for this task. Our numerical results show substantial improvement in the multifidelity error per simulation cost. The GOMP algorithm was compared to previous allocation strategies as well as existing column subset selection problem solutions on several simulation datasets. This work thus investigates an unexplored connection between the multifidelity allocation problem and the classic subset selection problem in the machine learning and data mining domains. Our analysis provides a first step to understanding the apparent superiority of the GOMP strategy, but a more rigorous analysis requires further investigation.
}

\section*{Acknowledgments}
We would like to thank Dr. Vahid Keshavarzzadeh for providing the data and assistance for the structure topology optimization example, 
and to thank Mr. Ashok Jallepalli for assistance with generating the fluid simulation example.

\bibliographystyle{siamplain}
\bibliography{main}
\end{document}